\documentclass[12pt,a4paper]{article}
\usepackage{amsmath,amsthm,amssymb,times,titlesec}
\usepackage{graphicx}
\usepackage{pdfsync}
\usepackage{float}
\usepackage{xcolor}
\usepackage{subfigure}
\numberwithin{equation}{section}

\usepackage[font=small]{caption}

\textheight 
            650pt
\textwidth 
           460pt

\titleformat{\subsubsection}
{\normalfont\fontsize{13}{14}\selectfont\itshape}{\thesubsubsection}{1em}{}

\newcommand\scaleddot{\scalebox{.89}{.}}
\makeatletter
\renewcommand{\ddddot}[1]{%
  {\mathop{\kern\z@#1}\limits^{\makebox[0pt][c]{\vbox to-2\ex@{\kern-\tw@\ex@\hbox{\normalfont\scaleddot\kern-0.5pt\scaleddot\kern-0.5pt\scaleddot\kern-0.5pt\scaleddot}\vss}}}}}

\newcommand{\bl}{[{\mskip-2mu}[}
\newcommand{\br}{]{\mskip-2mu}]}
\newcommand{\abs}[1]{\lvert#1\rvert}
\newcommand\norm[1]{\left\lVert#1\right\rVert}

\newcommand{\R}{\mathbb{R}}
\newcommand{\Z}{\mathbb{Z}}

\newcommand{\ee}{\mathrm{e}}
\newcommand{\ii}{\mathrm{i}}

\newcommand{\ds}{\, \mathrm{d}s}
\newcommand{\dt}{\, \mathrm{d}t}
\newcommand{\dx}{\, \mathrm{d}x}

\newcommand{\dy}{\, \mathrm{d}y}
\newcommand{\dz}{\, \mathrm{d}z}

\newcommand{\DD}{{\mathcal D}}

\newcommand{\ZZ}{{\mathcal Z}}

\DeclareMathOperator{\cosech}{cosech}

\DeclareMathOperator{\cosec}{cosec}

\newtheorem{theorem}{Theorem}[section]
\newtheorem{lemma}[theorem]{Lemma}
\newtheorem{proposition}[theorem]{Proposition}
\newtheorem{corollary}[theorem]{Corollary}
\newtheorem{remarks}[theorem]{Remark}

\allowdisplaybreaks

\newcounter{count}

\title{A resonant Lyapunov centre theorem with \\an application to doubly periodic\\ travelling hydroelastic waves}

\author{R. Ahmad\thanks{FR Mathematik, Universit\"{a}t des Saarlandes, Postfach 151150, 66041 Saarbr\"{u}cken, Germany}
\and
M. D. Groves\footnotemark[1]
\and
D. Nilsson\thanks{Centre for Mathematical Sciences, Lund University, Box 118, 221 00 Lund, Sweden}}
\date{}

\begin{document}

\maketitle

\begin{abstract}
We present a Lyapunov centre theorem for an antisymplectically reversible Hamiltonian system exhibiting a nondegenerate
$1:1$ or $1:-1$ semisimple resonance as a detuning parameter is varied.
The system can be finite- or infinite dimensional (and quasilinear) and have a non-constant symplectic structure. We allow the 
origin to be a `trivial' eigenvalue arising from a translational symmetry or, in an infinite-dimensional setting, to lie in the continuous
spectrum of the linearised Hamiltonian vector field provided a compatibility condition on its range is satisfied.

As an application we show how Kirchg\"{a}ssner’s spatial dynamics approach can be used to construct doubly periodic travelling
waves on the surface of a three-dimensional body of water (of finite or infinite depth) beneath a thin ice sheet (`hydroelastic waves'). The
hydrodynamic problem is formulated as a reversible Hamiltonian system in which an arbitrary horizontal spatial direction is the time-like variable
and the infinite-dimensional phase space consists of wave profiles which are periodic (with fixed period) in a second, different horizontal direction.
Applying our Lyapunov centre theorem at a point in parameter space associated with a $1:1$ or $1:-1$ semisimple resonance yields
a periodic solution of the spatial Hamiltonian system corresponding to a doubly periodic hydroelastic wave.
\end{abstract}

\section{Introduction}

\subsection{Resonant Hamiltonian systems} \label{Overview}

A linear dynamical system
$$\dot{v}=Lv$$
for which $L \in {\mathbb R}^{2n \times 2n}$ has a pair of simple, purely imaginary eigenvalues $\pm \ii \omega$ has a periodic orbit with frequency $\omega$.
The classical \emph{Lyapunov centre theorem} asserts that a (smooth) nonlinear perturbation
\begin{equation}
\dot{v} = Lv + N(v) \label{intro - DS}
\end{equation}
of this dynamical system has a family of small-amplitude periodic solutions with frequency near $\omega$
provided that (i) it is Hamiltonian or reversible, and (ii) the non-resonance condition that $\ii n \omega$ is not
an eigenvalue for any integer $n \neq \pm 1$ is satisfied (see Kielh\"{o}fer \cite[\S\S I.11.1--I.11.2]{Kielhoefer}).
The theorem can be extended to infinite-dimensional, quasilinear systems under the nonresonance condition
that $\ii n \omega \not\in \sigma(L)$ for $n \neq \pm 1$,  and furthermore the condition that $0 \not\in \sigma(L)$ can also
be replaced by a compatibility condition on the range of $L$ (see Iooss \cite{Iooss99}).

In this article we consider a parameter-dependent evolutionary equation of the form
$$v_t = L^{\mu_1}v + N^{\mu_1}(v)$$
which is both reversible and Hamiltonian; it may be finite- or infinite-dimensional (and quasilinear).
The linear operator $L^{\mu_1}$ is supposed to have two pairs
$\pm \ii \kappa_1^{\mu_1}$, $\pm \ii \kappa_2^{\mu_1}$ of simple purely imaginary eigenvalues
which depend smoothly upon $\mu_1$ and collide with `non-zero speed' at $\mu_1=0$, that is
$$\kappa_1^0=\kappa_2^0, \qquad
\frac{\mathrm{d}}{\mathrm{d}\mu_1} (\kappa_1^{\mu_1}-\kappa_2^{\mu_1})\Big|_{\mu_1=0} \neq 0.$$
The collision is semisimple, that is at criticality the eigenvalues $\pm \ii \kappa$, where $\kappa=\kappa_1^0=\kappa_2^0$,
are geometrically and algebraically double. We assume the nonresonance condition $\ii n \kappa \not\in \sigma(L^0)$ for $n=\pm 2,\pm 3,\ldots$
but allow the origin to be a `trivial' eigenvalue arising from a translational symmetry or, in an infinite-dimensional setting, to lie in the continuous
spectrum of $L$ provided a compatibility condition on its range is satisfied. The result is a two-parameter family
$\{v(t_1,t_2),\mu_1(t_1,t_2)\}_{0 \leq t_1,t_2 < \varepsilon}$ of $2\pi/(\kappa +\mu_2(t_1,t_2))$-periodic reversible solutions,
where $\mu_1(t_1,t_2), \mu_2(t_1,t_2) \rightarrow 0$ along with the amplitude of the solutions as $(t_1,t_2) \rightarrow (0,0)$.

A classical Hamiltonian system with $n$ degrees of freedom has the form
$$\dot{v}=J\nabla H(v), \qquad J = \begin{pmatrix} 0 & I \\ -I & 0 \end{pmatrix},$$
where the Hamiltonian $H$ is a smooth real-valued function of $v \in {\mathbb R}^{2n}$ and $0$, $I$ denote the $n \times n$ zero and identity matrices;
in the above notation $L=J \nabla H_2$, where $H_2(v)=\frac{1}{2}\mathrm{d}^2H[0](v,v)$. Examining a 
two-degree-of-freedom Hamiltonian system whose
linearisation has two semisimple eigenvalues $\pm \ii \kappa$, one finds that
$H_2(q_1,q_2,p_1,p_2)=\frac{1}{2}\kappa(q_1^2+p_1^2) \pm \frac{1}{2}\kappa(q_2^2+p_2^2)$, the two cases of which are referred to as a
$1:1$ and $1:-1$ resonance respectively. (This terminology is also applied to higher-order Hamiltonian systems with no other eigenvalue resonances
by restricting to the eigenspaces corresponding to $\pm \ii \kappa$.)
Periodic solutions of two-degree-of-freedom Hamiltonian systems with semisimple $1:1$ and $1:-1$ resonances were studied by Kummer \cite{Kummer76,Kummer78},
while the corresponding non-semisimple `Hamiltonian-Hopf' resonance (in which $\pm \ii \kappa$ are geometrically simple and algebraically double eigenvalues) was studied by van der Meer \cite{Vandermeer}.

The dynamical system \eqref{intro - DS} is \emph{reversible} if there exists an involution $R$ (the `reverser') which anticommutes with $L$ and $N$. A reversible
system has the property that $Ru$ is a solution whenever $u$ is a solution; a solution $u$ is termed \emph{reversible} or \emph{symmetric} if
it is invariant under $R$. Reversible Hamiltonian systems have the property that either $H(Rv)=H(v)$ or $H(Rv)=-H(v)$; these cases are referred to
as `antisymplectic' and `symplectic' respectively. Small-amplitude periodic solutions of symplectically reversible, $n$-degree-of-freedom Hamiltonian systems
which exhibit a semisimple $1:-1$ resonance were studied by Alomair \& Montaldi \cite{AlomairMontaldi17} (see also
Buzzi \& Lamb \cite{BuzziLamb05a}) using Lyapunov-Schmidt reduction.

In the present paper we consider parameter-dependent, antisymplectically reversible Hamiltonian systems of the form
\begin{equation}\label{intro_main_eq}
v_t = J^{\mu_1}(v)\nabla H^{\mu_1}(v)
\end{equation}
which exhibit a semisimple $1:1$ or $1:-1$ resonance in the eigenvalues $\pm \ii \kappa$ when $\mu_1=0$. The system may be finite- or infinite-dimensional, and 
$J^{\mu_1}(v)$ is an invertible linear operator which is skew-symmetric with respect to a suitable inner product $\langle \cdot\,,\cdot \rangle$ and,
as the notation indicates, is not necessarily constant (see below for a precise statement).
We look for periodic solutions of \eqref{intro_main_eq} with frequency near $\kappa$ by writing
$$v(t)=u(\tau), \qquad \tau = (\kappa +\mu_2)t$$
and seeking $2\pi$-periodic solutions of the transformed equation. For this purpose we use variational Lyapunov-Schmidt reduction,
seeking critical points of the functional
$$
S(u,\mu_1,\mu_2)=\frac{1}{2\pi}\int_0^{2\pi}\bigg\{-(\kappa +\mu_2)\langle\alpha^{\mu_1}(u),u_\tau\rangle -H^{\mu_1}(u)\bigg\}\ \mathrm{d}\tau,
$$
where $\alpha^{\mu_1}(v)$ is an anti-derivative of $J^{\mu_1}(v)$, in a suitable `loop space'. A precise statement of our theorem is given in
Section \ref{Intro - theorem} below; the proof is presented in Section \ref{proof of LCT}.

\subsubsection*{Hamiltonian formalism}

Let $X$, $Z$ be real Hilbert spaces, where $X$ is continuously and densely embedded in $Z$, and 
$Z$ is equipped with an additional continuous inner product $\langle \cdot\,,\cdot \rangle$ which does not necessarily induce its strong topology.
Let $\Lambda_1 \times U$ be a neighbourhood of the origin in ${\mathbb R} \times X$.
We regard $U$ as a manifold domain of $Z$ by extending elements of the tangent space $TX|_v \cong X^\ast \cong X$
of $X$ at the point $v \in U$ to elements of the tangent space $TZ|_v \cong Z^\ast \cong Z$ of $Z$ at this point. The derivative
${\mathrm{d}F}^{\mu_1}[v] \in X^\ast$ of a smooth real-valued function $F^{\mu_1}(v)$ of $({\mu_1},v) \in \Lambda_1 \times U$ 
has a unique extension $\widetilde{\mathrm{d}F}{\vphantom{}}^{\mu_1}[v] \in Z^\ast$.
We use the same notation for smooth functions $F^{\mu_1}(v)$ of $({\mu_1},v)$ with values in $X$, so that $\widetilde{\mathrm{d}F}\vphantom{}^{\mu_1}[v] \in {\mathcal L}(Z)$, and occasionally assume that there exists an adjoint operator
$\widetilde{\mathrm{d}F}\vphantom{}^{\mu_1}[v]^\ast \in {\mathcal L}(Z)$ such that
$$\langle \widetilde{\mathrm{d}F}\vphantom{}^{\mu_1}[v]^\ast(v_1), v_2 \rangle = \langle v_1, \widetilde{\mathrm{d}F}\vphantom{}^{\mu_1}[v](v_2)\rangle$$
for all $(\mu_1,v) \in \Lambda_1 \times U$ and $v_1$, $v_2 \in Z$. Both $\widetilde{\mathrm{d}F}\vphantom{}^{\mu_1}[v]$ and
$\widetilde{\mathrm{d}F}\vphantom{}^{\mu_1}[v]^\ast$ are assumed to
depend smoothly upon $({\mu_1},v) \in \Lambda_1 \times U$.

Using this framework we make the following definitions and assumptions.
\begin{itemize}
\item[(i)]
\emph{A (parameter-dependent) $k$-form on $U$} is an alternating, bounded, $k$-linear mapping\linebreak
$Z^k \rightarrow {\mathbb R}$
which depends smoothly upon $({\mu_1},v) \in \Lambda_1 \times U$. An \emph{exact symplectic $2$-form}
$\Omega^{\mu_1}|_v$ on $U$ is a $2$-form
given by
$$
\Omega^{\mu_1}|_v(v_1,v_2) = \langle J^{\mu_1}(v) v_1,v_2 \rangle, \label{First formula for Omega}
$$
where $J^{\mu_1}(v)$ is an invertible, skew-symmetric linear mapping in ${\mathcal L}(Z)$ which depends smoothly upon $({\mu_1},v) \in \Lambda_1 \times U$. Furthermore
$\Omega^{\mu_1}|_v$ is the exterior derivative of a $1$-form $\omega^{\mu_1}|_v$ given by
$$\omega^{\mu_1}|_v(w) = \langle \alpha^{\mu_1}(v),w \rangle,$$
where $\alpha^{\mu_1}(v)$ is an element of $Z$ which depends smoothly upon $({\mu_1},v) \in \Lambda_1 \times U$; in other words
$$
\Omega^{\mu_1}|_v(v_1,v_2) = \langle \widetilde{\mathrm{d}\alpha}\vphantom{}^{\mu_1}[v](v_1),v_2 \rangle - \langle v_1,\widetilde{\mathrm{d}\alpha}\vphantom{}^{\mu_1}[v](v_2)\rangle,
$$
so that
$$J^{\mu_1}(v) = \widetilde{\mathrm{d}\alpha}\vphantom{}^{\mu_1}[v] - \widetilde{\mathrm{d}\alpha}\vphantom{}^{\mu_1}[v]^\ast$$
\item[(ii)]
\emph{A (parameter-dependent)
Hamiltonian on $U$} is a smooth real-valued function $H^{\mu_1}(v)$ of $({\mu_1},v) \in \Lambda_1 \times U$
which satisfies $H^{\mu_1}(0)=0$ and $\mathrm{d}H^{\mu_1}[0]=0$ for all $\mu_1 \in \Lambda_1$.
We assume that its gradient, that is the element $\nabla H^{\mu_1}(v)$ of $Z$ with 
$$\widetilde{\mathrm{d}H}\hspace{-4mm}{\hphantom{H}}^{\mu_1}[v](w) = \langle \nabla H^{\mu_1}(v),w\rangle$$
for all $w \in Z$, exists for all $v$ in a dense subset $\DD_\mathrm{H}$ of $\Lambda_1 \times U$ and extends to a smooth function of
$(\mu_1,v) \in \Lambda_1 \times U$.
\item[(iii)]
\emph{The Hamiltonian vector field} $v_\mathrm{H}^{\mu_1}$ of a \emph{Hamiltonian system} $(Z,\Omega^{\mu_1}, H^{\mu_1})$, where
$\Omega^{\mu_1}$ is an exact symplectic $2$-form and $H^{\mu_1}$ is a Hamiltonian on $U$, is given by
$$v_\mathrm{H}^{\mu_1}(v)=J^{\mu_1}(v)^{-1}\nabla H^{\mu_1}(v), \qquad (\mu_1,v) \in \DD_\mathrm{H};$$
it yields the unique element $v_\mathrm{H}^{\mu_1}(v)$
of $Z$ such that
$$\Omega^{\mu_1}|_v(v_\mathrm{H}^{\mu_1}(v),w)=\widetilde{\mathrm{d}H}\vphantom{}^{\mu_1}[v](w)$$
for all $w \in Z$. \emph{Hamilton's equations}
$$v_t=v_\mathrm{H}^{\mu_1}(v), \qquad (\mu_1,v) \in \DD_\mathrm{H},$$
determine the orbits generated by the Hamiltonian vector field.
\end{itemize}

\subsection{The main result} \label{Intro - theorem}

Let $X$, $Z$ be real Hilbert spaces, where $X$ is continuously and densely embedded in $Z$, and
consider the autonomous evolutionary equation
\begin{equation}
v_t = L^{\mu_1}v + N^{\mu_1}(v), \label{Eqn in LI theorem}
\end{equation}
where $({\mu_1},v) \mapsto N^{\mu_1}(v)$ is a smooth mapping from
a neighbourhood $\Lambda_1 \times U$ of the origin in ${\mathbb R} \times X$ into $Z$ which satisfies $N^{\mu_1}(0)=0$ and $\mathrm{d}N^{\mu_1}[0]=0$ for all $\mu_1 \in \Lambda_1$, and
$L^{\mu_1}: X \subseteq Z \rightarrow Z$ is a closed linear operator which depends smoothly upon $\mu_1$. We study equation \eqref{Eqn in LI theorem}
under the following hypotheses.

\begin{itemize}
\item[(H1)] 
Equation \eqref{Eqn in LI theorem} represents Hamilton's equations
$$v_t = v_\mathrm{H}^{\mu_1}(v)$$
for a Hamiltonian system $(Z,\Omega^{\mu_1}, H^{\mu_1})$ with $\DD_\mathrm{H}=\Lambda_1 \times U$ (in the above nomenclature), so that
$$L^{\mu_1}(v) + N^{\mu_1}(v)=J^{\mu_1}(v)^{-1}\nabla H^{\mu_1}(v)$$
and in particular
$$L^{\mu_1}v=J^{\mu_1}(0)^{-1}\nabla H_2^{\mu_1}(v),$$
where $H_2^{\mu_1}(v)=\frac{1}{2}\mathrm{d}^2H^{\mu_1}[0](v,v)$ (the part of $H^{\mu_1}$ which is homogeneous of degree $2$ in $v$).
\item[(H2)] Equation \eqref{Eqn in LI theorem} is reversible: both $L^{\mu_1}$ and $N^{\mu_1}$ anticommute with an involution $R \in {\mathcal L}(X) \cap {\mathcal L}(Z)$.
This \emph{reverser} $R$ satisfies
$$H^{\mu_1}(Rv) = H^{\mu_1}(v), \quad R^\ast\alpha^{\mu_1}(Rv)= -\alpha^{\mu_1}(v), \quad R^\ast J^{\mu_1}(Rv)R =-J^{\mu_1}(v)$$
for all $(\mu_1,v) \in \Lambda_1 \times U$.
\end{itemize}

\noindent There are also spectral hypotheses on $L^{\mu_1}$.

\begin{itemize}
\item[(H3)] The linear operator $L^{\mu_1}$ has two pairs $\pm \ii \kappa_1^{\mu_1}$, $\pm \ii \kappa_2^{\mu_1}$ of purely imaginary eigenvalues with linearly independent eigenvectors $e_1^{\mu_1}$, $\bar{e}_1^{\mu_1}$, $e_2^{\mu_1}$, $\bar{e}_2^{\mu_1}$, all of which depend smoothly upon $\mu_1$.
Furthermore $\bar{e}_1^0=Re_1^0$, $\bar{e}_2^0=Re_2^0$ and
$$\kappa_1^0=\kappa_2^0, \qquad
\frac{\mathrm{d}}{\mathrm{d}\mu_1} (\kappa_1^{\mu_1}-\kappa_2^{\mu_1})\Big|_{\mu_1=0} \neq 0.$$
For notational simplicity we henceforth abbreviate $L^0$, $e_1^0$ and $e_2^0$ to respectively $L$, $e_1$ and $e_2$,
and define $\kappa:=\kappa_1^0=\kappa_2^0$.

\item[(H4)]The origin is one of
\begin{itemize}\vspace{-0.5\baselineskip}
\item[(i)]
a point of the resolvent set of $L$,
\item[(ii)]
a point of the continuous spectrum of $L$,
\item[(iii)]
a geometrically simple, algebraically double eigenvalue of $L$
with eigenvector $f_1$ and generalized eigenvector $f_2$ (possibly embedded in continuous spectrum),
where $Rf_1=- f_1$ and $Rf_2= f_2$.
\end{itemize}
Here (ii) and (iii) entail further spectral hypotheses (see (H7) and (H8) below).
\item[(H5)] The imaginary number $\ii k\kappa $ belongs to the resolvent set of $L$ for each $k\in \mathbb{Z}\backslash \{0,-1,1\}$.
\item[(H6)] The linear operator $L$ satisfies
$$
\norm{(\ii k\kappa I-L)^{-1}}_{Z\rightarrow Z}\lesssim \frac{1}{\abs{k}},\qquad
\norm{(\ii k\kappa I-L)^{-1}}_{Z\rightarrow X}\lesssim 1
$$
for all $k\in \mathbb{Z}\backslash \{0,-1,1\}$.
\item[(H7)]
The zero eigenvalue (if present) is `trivial' in the following sense: writing $u \in U$ as $u=q f_1 + w$ with $w \in \{f_1\}^\perp$, one finds that $J^{\mu_1}(u)$ and $H^{\mu_1}(u)$ do not depend upon $q$.
\item[(H8)] 
Suppose that the equation
\[
Lu^\dag=J^0(0)^{-1}(I-\Pi_0)N^\star(u,\mu_1,\mu_2),
\]
where
$$N^\star(u,\mu_1,\mu_2)
=\frac{1}{\sqrt{2\pi}}\int_0^{2\pi} \Big(J^{\mu_1}(u)\big((\kappa +\mu_2)J^{\mu_1}(u)u_\tau-L^{\mu_1}u-N^{\mu_1}(u)\big)+J^0(0)Lu\Big)\, \mathrm{d}\tau,$$
has a unique solution $u^\dag \in (I-\Pi_0)X$ which depends smoothly upon $u \in {\mathcal U}$, $\mu_1 \in \Lambda_1$ and
$\mu_2 \in \Lambda_2$, where $\Pi_0$ is the orthogonal projection of $X$ onto $\mathrm{span}\{f_1,f_2\}$ and
\begin{equation*}
 {\mathcal U}=\{u \in H_{\mathrm{per}}^1(\mathbb{R},Z)\cap L_\mathrm{per}^2(\mathbb{R},X): \mbox{$u(\tau) \in U$ for all $\tau \in {\mathbb R}$}\}.
\end{equation*}
\end{itemize}

\begin{remarks} \hspace{1cm}
\begin{itemize}
\item[(i)]
Hypothesis (H8) is meaningful only if the origin lies in the continuous spectrum of $L$ or is an eigenvalue embedded in continuous spectrum, since it is
automatically satisfied if the origin lies in the resolvent set of $L$ or is an isolated eigenvalue of $L$.
\item[(ii)] If $J^{\mu_1}(u)$ is constant and $\Pi_0=0$, then
$$N^\star(u,\mu_1,\mu_2)=-\frac{1}{\sqrt{2\pi}}\int_0^{2\pi}N^{\mu_1}(u)\,\mathrm{d}\tau,$$
so that hypothesis (H8) reduces to the condition used by Iooss \cite{Iooss99}.
\end{itemize}
\end{remarks}

\begin{theorem} \label{LCT}
Under hypotheses (H1)--(H8) there exist $\varepsilon>0$ and a smooth, two-parameter branch
$\{v(t_1,t_2),\mu_1(t_1,t_2)\}_{0 \leq t_1,t_2 < \varepsilon}$ of $2\pi/(\kappa +\mu_2(t_1,t_2))$-periodic reversible solutions
to equation \eqref{Eqn in LI theorem} in $H_{\mathrm{loc}}^1(\mathbb{R},Z)\cap L_\mathrm{loc}^2(\mathbb{R},X)$.
The rescaled function $v(t)=u(\tau)$, $\tau=(\kappa +\mu_2)t$ satisfies $\|u(t_1,t_2)\|_\ZZ \rightarrow 0$, while
$\mu_1(t_1,t_2), \mu_2(t_1,t_2) \rightarrow 0$ as $(t_1,t_2) \rightarrow (0,0)$.
\end{theorem}

Theorem \ref{LCT} is proved in Section \ref{proof of LCT}.

\subsection{Hydroelastic waves}

In Section \ref{Formulation} we introduce the hydrodynamic problem for travelling waves on the surface of a three-dimensional body of water beneath a thin ice sheet modelled using the Cosserat theory of
hyperelastic shells (Plotnikov \& Toland \cite{PlotnikovToland11}). The fluid is bounded below by a rigid horizontal bottom $\{{x_2}=-h\}$ (the cases $h<\infty$
and $h=\infty$ are referred to as `finite depth' and `infinite depth') and
above by a free surface $\{{x_2}=\eta({x_1},{x_3})\}$ (in a frame of reference following the wave with constant speed $c$ in the ${x_1}$ direction); there is no cavitation between this surface and the ice sheet. The hydrodynamic problem is formulated in terms of $\eta$ and an Eulerian velocity potential $\phi$ in dimensionless form
in Section \ref{Formulation}; the govering equations \eqref{gov 1}--\eqref{gov 4} depend upon two dimensionless parameters
$$\beta=\left(\frac{D}{\rho g h^4}\right)^{\!\!1/4} \geq 0, \qquad \gamma=\left(\frac{c^8 \rho}{D g^3}\right)^{\!\!1/8}>0,$$
where $D$, $\rho$ and $g$
are respectively the coefficient of flexural rigidity for the ice sheet, the density of the fluid and the acceleration due to gravity
(see Guyenne and Parau \cite{GuyenneParau12}). The dimensionless fluid domain is $\{-\frac{1}{\beta} < x_2 < \eta(x_1,x_2)\}$, so that the limit $\beta \rightarrow 0$
corresponds to `infinite depth'.

We consider waves which are periodic with periods $p_1$ and $p_2$ in two arbitrary horizontal directions $x$ and $z$
which form (different) angles $\theta_1$, $\theta_2 \in [0,\pi)$ with the ${x_1}$-axis respectively, so that
\[
x=\csc(\theta_2-\theta_1)({x_1}\sin\theta_2-{x_3}\cos\theta_2), \qquad
z=\csc(\theta_1-\theta_2)({x_1}\sin\theta_1-{x_3}\cos\theta_1)
\]
(see Figure \ref{intro - periodic domain}). 
To this end we seek solutions of the governing equations of the form
$$
\eta({x_1},{x_3})=\tilde{\eta}(\tilde{x},\tilde{z}),
$$
where
\begin{equation*}
\tilde{x} = {x_1}\sin\theta_2-{x_3}\cos\theta_2, \qquad
\tilde{z} = \nu({x_1}\sin\theta_1-{x_3}\cos\theta_1)
\end{equation*}
with $\nu=2\pi/p_2$ and $\tilde{\eta}$ is $2\pi$-periodic in $\tilde{z}$.  We proceed by formulating the hydroelastic problem as a reversible
Hamiltonian system in which the horizontal spatial direction $\tilde{x}$ plays the role of the time-like variable (`spatial dynamics'), working
in a phase space of functions which are $2\pi$-periodic in $\tilde{z}$.  By construction a periodic (in `time') solution of the evolutionary system,
corresponds to a surface wave which is periodic in both $\tilde{x}$ and $\tilde{z}$ and is found using a Lyapunov centre theorem
(although care is required in interpreting such solutions; see below).

To formulate the hydroelastic equations as an evolutionary system, we exploit the observation
that they follow from a variational principle (a modified version of the classical variational principle for water waves introduced by Luke \cite{Luke67}). We treat the variational functional as an action functional in which a density is integrated over the $\tilde{x}$ direction and perform a Legendre transform (which is actually higher order due to the presence of 
second-order derivatives in the density) to derive a formulation of the hydrodynamic problem as an infinite-dimensional, reversible Hamiltonian system in which $\tilde{x}$ is the time-like variable. Although this procedure is formal, it delivers a candidate for a formulation of the hydrodynamic problem as an evolutionary system whose mathematical correctness is readily confirmed a posteriori; full details are given in Section \ref{Formulation}. Finally, we introduce a bifurcation parameter $\mu_1$ by
writing $\nu=\nu_0+\mu_1$, where $\nu_0$ is a reference value for $\nu$ (see below) and use a change of variable to linearise a nonlinear boundary condition emerging from the Legrendre transform. The result is a system of the form
\begin{equation}
\hat{v}_x = L^{\mu_1} \hat{v} + N^{\mu_1}(\hat{v}) \label{intro - SD}
\end{equation}
which is amenable to Theorem \ref{LCT}, satisfying (H1) and (H2) by construction.
\begin{figure}
\centering
\includegraphics[scale=0.5]{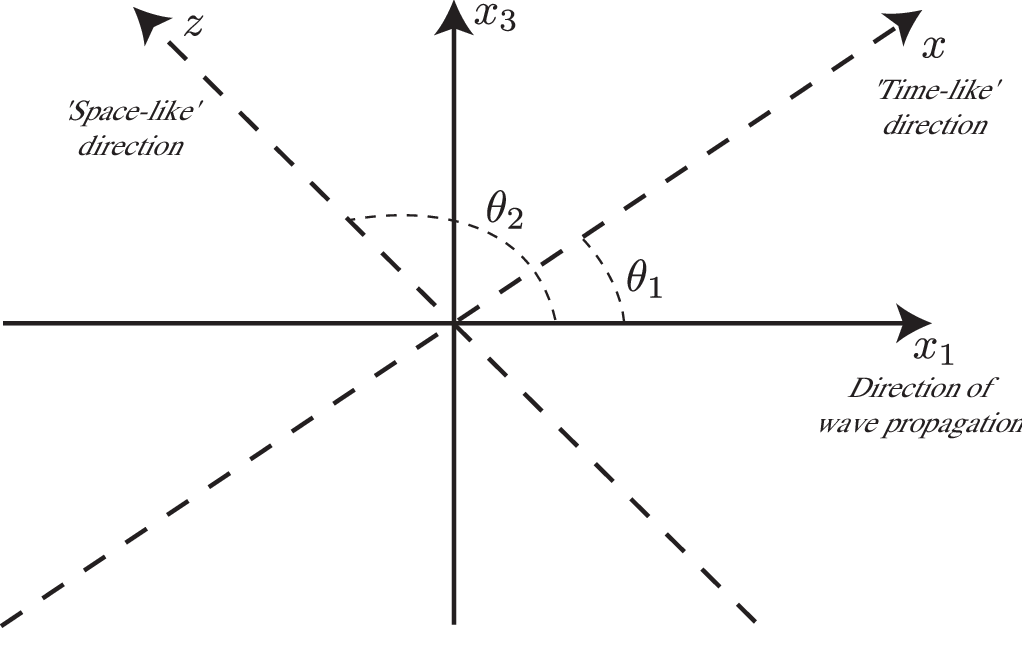}
\caption{In the spatial dynamics formulation of the hydrodynamic problem the $x$ and $z$ directions are treated as respectively ‘time-like’ and ‘space like’.}
\label{intro - periodic domain}
\end{figure}

A purely imaginary eigenvalue $\ii s$ of $L:=L^0$ with corresponding eigenvector in the $k$th Fourier mode (where $(k,s) \neq (0,0)$) corresponds to a linear hydroelastic wave of the form
$\eta(\tilde{x},\tilde{z})=\eta_{s,k} \ee^{\ii \ell_1 \tilde{x} + \ii \ell_2 \tilde{z}}$ with
\pagebreak
\begin{align}
\ell_1&=s\sin \theta_2+\nu_0k\sin \theta_1,\label{intro - l1eq}\\
\ell_2&=-s\cos\theta_2 -\nu_0 k \cos\theta_1,\label{intro - l2eq}
\end{align}
and a solution of this kind exists if and only if $\ell_1$ and $\ell_2$ satisfy the dispersion relation
$$
\DD(\ell_1,\ell_2):=(1+(\ell_1^2+\ell_2^2)^2)\sqrt{\ell_1^2+\ell_2^2}\tanh\left( \beta^{-1}\sqrt{\ell_1^2+\ell_2^2}\right)-\gamma^2\ell_1^2=0.
$$
A mode $k$ purely imaginary eigenvalue $\ii s$  therefore corresponds to an intersection in the $(\ell_1,\ell_2)$-plane of
the dispersion curve
$$C_\mathrm{dr}=\{(\ell_1,\ell_2) \in {\mathbb R}^2 \setminus \{(0,0)\}: \DD(\ell_1,\ell_2)=0\}$$
with the straight line $S_k$ defined by equations \eqref{intro - l1eq}, \eqref{intro - l2eq}. (The solution $(\ell_1,\ell_2)=(0,0)$ of $\DD(\ell_1,\ell_2)=0$ is excluded since it corresponds to $(k,s)=(0,0)$.) The dispersion curve is
sketched in Figure \ref{Intersections}(a).
The $(\beta,\gamma)$-parameter plane is divided into three regions in which $C_\mathrm{dr}$ has respectively zero, one, and two nontrivial bounded branches; the delineating curves $D_1$ and $D_2$ are given explicitly in Section \ref{pis}.
The central region is in fact each divided into two subregions, at the mutual boundary of which the qualitative shape of the branches changes, namely, from convex to nonconvex. 

A point of intersection of $S_k$ and $C_\mathrm{dr}$ corresponds to a purely imaginary mode $k$ eigenvalue $\ii s$; its imaginary part is the value of the parameter $s$ at the point of intersection, that is, the value of $S_0$ in the $(S_0,T)$-coordinate system at the intersection, where
\[T=\{(\ell_1,\ell_2)\in {\mathbb R}^2\;:\;
\ell_1=\sin\theta_1\,a,\
\ell_2=-\cos\theta_1\,a,\ a\in\R\}\]
(see Figure \ref{Intersections}(b)).
The geometric multiplicity of the eigenvalue $\ii s$ is given by the number of distinct lines in the family $\{S_k\}$ which intersect $C_\mathrm{dr}$ at this parameter value, and a tangent intersection between $S_k$ and $C_\mathrm{dr}$ indicates that each eigenvector in mode $k$ has an associated Jordan chain of length at least $2$.
Notice that the sets
$S_k \cap C_\mathrm{dr}$ and $S_{-k} \cap C_\mathrm{dr}$ have the same cardinality:
the purely imaginary number $\ii s$ is a mode $k$ eigenvalue
if and only if the purely imaginary number $-\ii s$ is a mode $-k$ eigenvalue.

\begin{figure}[h!]\centering
\includegraphics{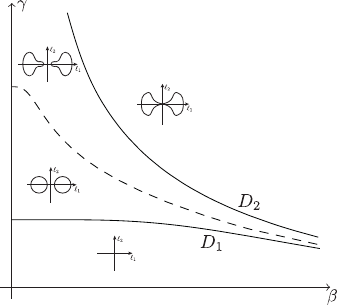}\hspace{2cm}
\includegraphics[scale=0.55]{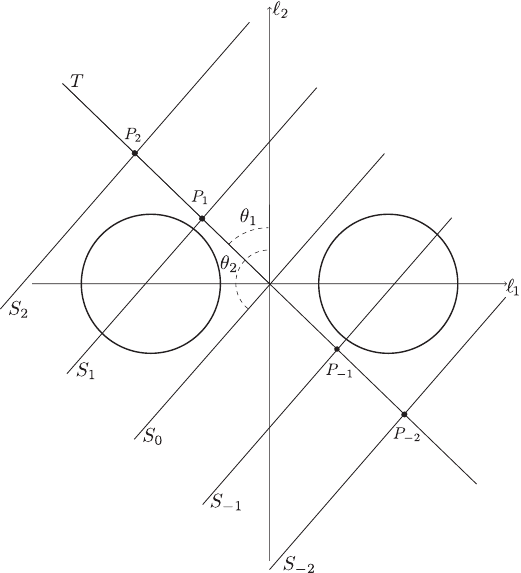}
\caption{(a) Shape of the dispersion curve $C_\mathrm{dr}$ in the $(\beta,\gamma)$-parameter plane. (b) Position of the lines $S_k$ and the dispersion curve $C_\mathrm{dr}$ in the $(\ell_1,\ell_2)$-plane;
the lines $S_0$ and $T$ form rangles $\theta_2$ and $\theta_1$ respectively with the positive $\ell_2$ axis, and the line $S_k$ intersects the line $T$ at the point $P_k$.} 
\label{Intersections}
\end{figure}

Let us now consider how to apply a Lyapunov centre theorem to equation \eqref{intro - SD}. A first attempt might be to 
choose $\beta$, $\gamma$ and $\theta_2$ such that $S_0$ does not intersect $C_\mathrm{dr}$ and $\theta_1$ such that $S_1$ and $S_{-1}$ intersect $C_\mathrm{dr}$ in points
with coordinates $(s,\nu_0)$, $(t,\nu_0)$ and $(-s,\nu_0)$, $(-t,\nu_0)$ in the $(S_0,T)$-coordinate system respectively,  while $S_k$ does not intersect 
$C_\mathrm{dr}$ for $k = \pm 2$, $\pm 3$, \ldots (see Figure \ref{Intersections}(b)). In this configuration $L$ has simple mode $1$ eigenvalues $\ii s$, $\ii t$
with eigenvectors $\hat{v}_{1,s}\ee^{\ii z}$, $\hat{v}_{1,t}\ee^{\ii z}$ and mode $-1$ eigenvalues $-\ii s$, $-\ii t$ with eigenvectors $\hat{v}_{-1,s}\ee^{-\ii z}$, $\hat{v}_{-1,t}\ee^{-\ii z}$.
Assuming that $0<t<s$ (and neglecting any spectrum at the origin for the moment), one
finds that the purely imaginary eigenvalues $\pm \ii s$ satisfy the non-resonance condition in a standard Lyapunov centre theorem, an infinite-dimensional version of which
therefore gives a periodic solution of \eqref{intro - SD} with period near $2\pi/s$. However this approach does not yield a genuinely three-dimensional wave: at the linear level the
solution takes the form $\hat{v} = \mathrm{Re}(\hat{v}_{1,s}\ee^{\ii s x}\ee^{\ii z})=\mathrm{Re}(\hat{v}_{1,s}\ee^{\ii (s x + z)})$, which depends upon the single spatial direction
$sx+z$. Note that Groves \& Haragus \cite{GrovesHaragus03} and Bagri \& Groves \cite{BagriGroves15}, while correctly elucidating the use of spatial dynamics and Lyapunov
centre theory to construct doubly periodic water waves, actually detect waves of this kind, often referred to as `$2 \frac{1}{2}$-dimensional waves'.

A more promising approach is to choose $\beta$, $\gamma$ and $\theta_2$ such that $S_0$ does not intersect $C_\mathrm{dr}$ and $\nu_0$ and $\theta_1$ such that $S_1$ and $S_{-1}$ each intersect $C_\mathrm{dr}$
in points with coordinates $(\pm s,\nu_0)$ and $(\pm s,-\nu_0)$ in the $(S_0,T)$-coordinate system, while $S_k$ does not intersect 
$C_\mathrm{dr}$ for $k = \pm 2$, $\pm 3$, \ldots (see Figure \ref{scenario}). In this configuration
$L$ exihibits a $1:1$ or $1:-1$ resonance: it has two mode $1$ eigenvalues $\pm \ii s$,
two mode $-1$ eigenvalues $\pm \ii s$, so that $\pm \ii s$ are geometrically and algebraically double eigenvalues of $L$ with eigenvectors
$\hat{v}_{1,s}\ee^{\ii z}$, $\hat{v}_{-1,s}\ee^{-\ii z}$ and $\hat{v}_{1,-s}\ee^{\ii z}$, $\hat{v}_{-1,-s}\ee^{-\ii z}$. Under the assumption that the
other hypotheses are satisfied, Theorem \ref{LCT} gives a periodic solution of \eqref{intro - SD} with period near $2\pi/s$ which yields a genuinely
three-dimensional wave: at the linear level the solution takes the form $\hat{v}=(\hat{v}_{1,s} \ee^{\ii s x}\ee^{\ii z} + \hat{v}_{1,-s} \ee^{-\ii s x}\ee^{-\ii z})$, which
cannot be reduced to a function of a single spatial direction. (One can of course apply this idea in any mode, assuming that $\pm \ii s$ are
both mode $k$ and mode $-k$ eigenvalues which do not resonate with any other purely imaginary eigenvalues.)

\begin{figure}[h!]\centering
\includegraphics[scale=0.55]{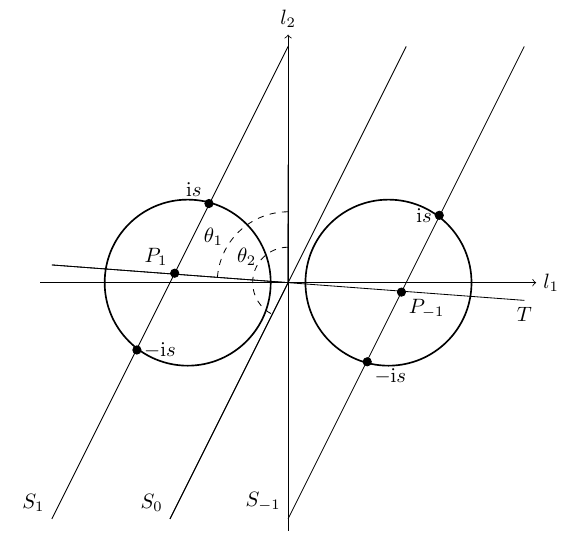}\includegraphics[scale=0.55]{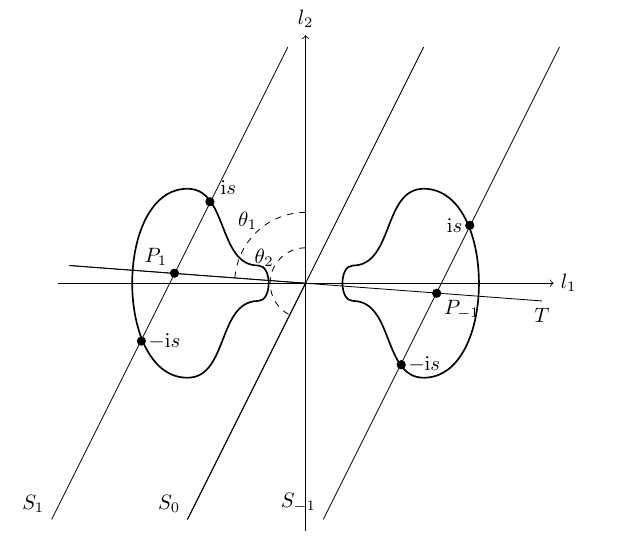}\includegraphics[scale=0.55]{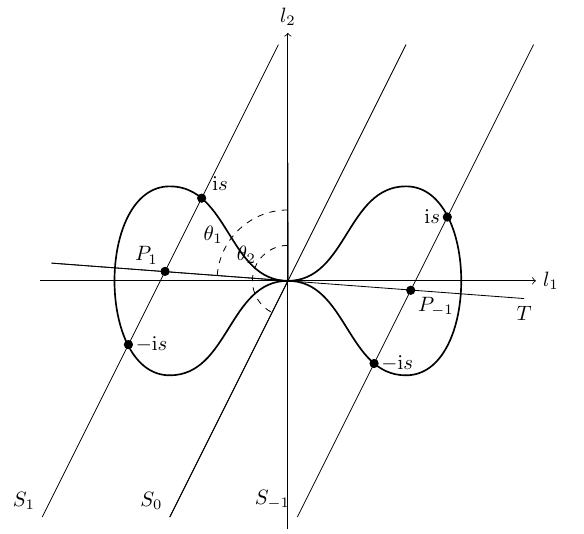}
\caption{Scenarios of interest.}
\label{scenario}
\end{figure}

In Section \ref{Application} we apply Theorem \ref{LCT} to the spatial dynamics formulation \eqref{intro - SD} of the hydro\-elastic problem in the eigenvalue scenarios shown in Figure \ref{scenario},
studying the purely imaginary spectrum of $L$ in detail in Section \ref{pis}. Hypotheses (H3), (H5) and (H6) are readily verified, while (H4)(iiii) (the origin is a double eigenvalue of $L$)
and (H4)(ii) (the origin is a point of the continuous spectrum of $L$) arise in the cases $\beta>0$ and $\beta=0$ respectively. In the former case we find that the zero eigenvalue is trivial in the
sense of (H7) and an isolated spectral point of $L$, so that (H8) is also satisfied. The additional verification of (H8) in the case $\beta=0$ is undertaken in Section \ref{Iooss condition}.
Altogether we establish the following result.

\begin{theorem} \label{First hydro result}
Choose $\beta$, $\gamma$ and $\theta_2$ such that $S_0$ does not intersect $C_\mathrm{dr}$ and $\nu_0$ and $\theta_1$ such that $S_1$ and $S_{-1}$ each intersect $C_\mathrm{dr}$
in points with coordinates $(\pm s,\nu_0)$ and $(\pm s,-\nu_0)$ in the $(S_0,T)$-coordinate system, while $S_k$ does not intersect 
$C_\mathrm{dr}$ for $k = \pm 2$, $\pm 3$, \ldots (see Figure \ref{scenario}).
There exist $\varepsilon>0$ and a two-parameter branch $\{(\phi,\eta)(t_1,t_2)\}_{0\leq t_1,t_2<\varepsilon}$ of doubly periodic solutions of \eqref{gov 1}--\eqref{gov 4} with
periods $2\pi/(s+\mu_2(t_1,t_2))$, $2\pi/(\nu_0+\mu_1(t_1,t_2))$ in the variables ${x_1}\sin\theta_2-{x_3}\cos\theta_2$ and ${x_1}\sin\theta_1-{x_3}\cos\theta_1$
respectively.
\end{theorem}

We also derive the following `inverse' result which shows that (under a nonresonance condition) one can find a family of doubly periodic solutions which are small perturbations of
any given periodic cell; these solutions have a fixed dimensionless wave speed $\gamma$ and fixed periodic directions.

\begin{theorem} \label{Second hydro result}
Choose $\beta$, $s$, $\nu_0$ and $\theta_2-\theta_1$ arbitrarily. There exist $\theta_1$ and $\gamma$ such that $S_1$ and $S_{-1}$ each intersect $C_\mathrm{dr}$
in points with coordinates $(\pm s,\nu_0)$ and $(\pm s,-\nu_0)$ in the $(S_0,T)$-coordinate system, so that Theorem \ref{First hydro result} holds under the additional hypothesis
that $S_k$ does not intersect  $C_\mathrm{dr}$ for $k \neq \pm 1$ .
\end{theorem}

\section{Proof of the main result} \label{proof of LCT}

In this section we prove Theorem \ref{LCT}, working in the framework set out in Section \ref{Intro - theorem} and under the hypotheses (H1)--(H8) given there.
We look for periodic solutions of \eqref{Eqn in LI theorem} with frequency near $\kappa $ by writing
$$v(t)=u(\tau), \qquad \tau = (\kappa +\mu_2)t,$$
where $\mu_2$ lies in a neighbourhood $\Lambda_2$ of the origin in ${\mathbb R}$, and seeking $2\pi$-periodic solutions
of the transformed equation
\begin{equation}
(\kappa +\mu_2)u_\tau = L^{\mu_1}u + N^{\mu_1}(u). \label{Scaled eqn in LI theorem}
\end{equation}
To this end we introduce the function spaces
\begin{align*}
{\mathcal X}&:=H_{\mathrm{per}}^1(\mathbb{R},Z)\cap L_\mathrm{per}^2(\mathbb{R},X),\\
{\mathcal Z}&:=L_\mathrm{per}^2(\mathbb{R},Z),
\end{align*}
equipping ${\mathcal Z}$ with the continuous scalar product
$$(\cdot\,,\cdot) = \frac{1}{2\pi}\int_0^{2\pi} \langle \cdot\,,\cdot \rangle$$
and noting that elements $w\in {\mathcal Z}$ can be expanded in Fourier series
\begin{equation*}
w(\tau)=\frac{1}{\sqrt{2\pi}}\sum_{k\in\mathbb{Z}}[w]_k \ee^{\ii k\tau},\qquad [w]_k=\frac{1}{\sqrt{2\pi}}\int_0^{2\pi}w(\tau)\ee^{-\ii k\tau}\ \mathrm{d}\tau \in Z.
\end{equation*}

We seek $2\pi$-periodic solutions of \eqref{Scaled eqn in LI theorem} by studying the function
$F:{\mathcal U}\times \Lambda_1\times \Lambda_2\mapsto {\mathcal Z}$ defined by
\begin{equation*}
F(u,\mu_1,\mu_2)=(\kappa +\mu_2)J^{\mu_1}(u)u_\tau-\nabla H^{\mu_1}(u),
\end{equation*}
where
$${\mathcal U}=\{u \in {\mathcal X}: \mbox{$u(\tau) \in U$ for all $\tau \in {\mathbb R}$}\}.$$
The equation
\begin{equation}
F(u,\mu_1,\mu_2)=0, \label{F eqn}
\end{equation}
has a variational characterisation.

\begin{proposition}\label{Euler-Lagrange_A}
Equation \eqref{F eqn} is the Euler-Lagrange equation for the action functional\linebreak
$S:{\mathcal U} \times \Lambda_1 \times \Lambda_2 \rightarrow {\mathbb R}$ given by
$$
S(u,\mu_1,\mu_2)=\frac{1}{2\pi}\int_0^{2\pi}\bigg\{-(\kappa +\mu_2)\langle\alpha^{\mu_1}(u),u_\tau\rangle -H^{\mu_1}(u)\bigg\}\ \mathrm{d}\tau.
$$
Furthermore $S$ is invariant with respect to the translation $T_\theta: u(\tau) \mapsto u(\tau+\theta)$ and
reversing operation $T: u(\tau) \mapsto (Ru)(-\tau)$.
\end{proposition}
\begin{proof}
The first assertion follows from the calculation
\begin{align*}
\mathrm{d}_1& S[u,\mu_1,\mu_2](v)\\
&=\frac{1}{2\pi}\int_0^{2\pi}\bigg\{-(\kappa +\mu_2)\big(\langle\widetilde{\mathrm{d}\alpha}\vphantom{}^{\mu_1}[u](v),u_\tau\rangle+\langle \alpha^{\mu_1}(u),v_\tau\rangle\big) -\langle\nabla H^{\mu_1}(u),v\rangle\bigg\}\ \mathrm{d}\tau\\
&=\frac{1}{2\pi}\int_0^{2\pi}\bigg\{-(\kappa +\mu_2)\big(\langle\widetilde{\mathrm{d}\alpha}\vphantom{}^{\mu_1}[u](v),u_\tau\rangle-\langle v,\widetilde{\mathrm{d}\alpha}\vphantom{}^{\mu_1}[u](u_\tau)\rangle\big)-\langle\nabla H^{\mu_1}(u),v\rangle\bigg\}\ \mathrm{d}\tau\\
&=\frac{1}{2\pi}\int_0^{2\pi}\langle (\kappa +\mu_2)J^{\mu_1}(u)u_\tau-\nabla H^{\mu_1}(u),v\rangle\ \mathrm{d}\tau \\
&=(F(u,\mu_1,\mu_2),v)
\end{align*}
for $v \in {\mathcal X}$, while the second is a consequence of the periodicity of $u$ and hypothesis (H2).
\end{proof}

The next step is a Lyapunov-Schmidt reduction. Define
\begin{align*}
{\mathcal W}_1&=\{u_0+Ae_1\ee^{\ii \tau}+Be_2\ee^{\ii \tau}+\bar{A}\bar{e}_1\ee^{-\ii \tau}+\bar{B}\bar{e}_2\ee^{-\ii \tau},\ A,B\in \mathbb{C},\ u_0\in Z\},\\
{\mathcal W}_2&=\{u\in {\mathcal Z}\colon [u]_0=0,\ \Pi_{\ii \kappa }[u]_1=\Pi_{-\ii \kappa }[u]_{-1}=0 \},
\end{align*}
where $\Pi_{\pm\ii \kappa }$ are the orthogonal projections onto the eigenspaces
$E_{\ii \kappa }=\mathrm{span}\{e_1,e_2\}$ and
$E_{-\ii \kappa }=\mathrm{span}\{\bar{e}_1,\bar{e}_2\}$, so that
\begin{align*}
{\mathcal Z}&={\mathcal W}_1\oplus {\mathcal W}_2,\\
{\mathcal X}&=({\mathcal W}_1\cap{\mathcal X})\oplus ({\mathcal W}_2\cap{\mathcal X})
\end{align*}
and the decompositions are orthogonal. Let $\widetilde{\Pi}_{{\mathcal W}_1}$ be the projection of ${\mathcal Z}$ onto ${\mathcal W}_1$ along ${\mathcal W}_2$, write
$u\in {\mathcal U}$ as 
\begin{equation*}
u=\underbrace{\widetilde{\Pi}_{{\mathcal W}_1} u}_{\displaystyle =:u_{{\mathcal W}_1}}+\underbrace{(I-\widetilde{\Pi}_{{\mathcal W}_1}) u}_{\displaystyle =:u_{{\mathcal W}_2}}
\end{equation*}
and equation \eqref{F eqn} as
\begin{align}
\widetilde{\Pi}_{{\mathcal W}_1} F(u_{{\mathcal W}_1}+u_{{\mathcal W}_2},\mu_1,\mu_2)&=0, \label{kernel-eq}\\
(I-\widetilde{\Pi}_{{\mathcal W}_1}) F(u_{{\mathcal W}_1}+u_{{\mathcal W}_2},\mu_1,\mu_2)&=0\label{range-eq}.
\end{align}

To solve equation \eqref{range-eq} (for $u_{{\mathcal W}_2}$ as a function of $u_{{\mathcal W}_1}$, $\mu_1$ and $\mu_2$)
 it is necessary to examine the solvability conditions for the equations
\begin{equation}
(\pm\ii \kappa  I - L)u = J^0(0)^{-1}f \label{Solve it 1}
\end{equation}
and
\begin{equation}
Lu = J^0(0)^{-1}f, \label{Solve it 2}
\end{equation}
where $f$ is a given function in $Z$. Normalise $e_1^{\mu_1}$, $e_2^{\mu_1}$ such that
$$\Omega^{\mu_1}|_0(e_1^{\mu_1},\bar{e}_1^{\mu_1})=\pm \ii, \quad \Omega^{\mu_1}|_0(e_2^{\mu_1},\bar{e}_2^{\mu_1}) = \pm \ii, \quad
\Omega^{\mu_1}|_0(e_1^{\mu_1},e_2^{\mu_1})=0, \quad \Omega^{\mu_1}|_0(e_1^{\mu_1},\bar{e}_2^{\mu_1})=0$$
and $f_1$, $f_2$ such that
$$\Omega^0|_0(f_1,f_2)=1,$$
where $\Omega^{\mu_1}|_0$ is extended \emph{bilinearly} to the complexification of $Z$. Observing that $L$ is the Hamiltonian vector field
for the linear Hamiltonian system $(Z,\Omega^0|_0,H_2^0)$, we find that the spectral projections $P_{\pm\ii \kappa}$ and $P_0$
onto the eigenspaces $E_{\ii \kappa }=\mathrm{span}\{e_1,e_2\},\ E_{-\ii \kappa }=\mathrm{span}\{\bar{e}_1,\bar{e}_2\}$ and generalised eigenspace $E_0=\mathrm{span}\{f_1,f_2\}$ are given by
\begin{align*}
P_{\ii \kappa } u &= \sum_{i=1}^2s_i\Omega^0|_0(u,\bar{e}_i)e_i= \sum_{i=1}^2 s_i \langle J^0(0)(u),e_i\rangle e_i \\
P_{-\ii \kappa }u&=-\sum_{i=1}^2s_i\Omega^0|_0(u,e_i)\bar{e}_i=- \sum_{i=1}^2 s_i \langle J^0(0)(u),\bar{e}_i\rangle\bar{e}_i,
\end{align*}
where $s_i=-\Omega^{\mu_1}|_0(e_i^{\mu_1},\bar{e}_i^{\mu_1})$,
and
\begin{align*}
P_0 u &= \Omega^0|_0(u,f_2)f_1 - \Omega^0|_0(u,f_1)f_2, \\
& = \langle J^0(0)(u),f_2 \rangle f_1 - \langle J^0(0)(u),f_1 \rangle f_2
\end{align*}
(see Mielke \cite[\S3.1]{Mielke}); here $\langle \cdot\,,\cdot \rangle$ is extended \emph{sesquilinearly} to the complexification of $Z$.
This observation shows in particular that the (necessary and sufficient) solvability condition for \eqref{Solve it 1},
namely that the spectral projection of its right-hand side onto $E_{\pm\ii\kappa }$ vanishes, is equivalent to the requirement
that the orthogonal projection $\Pi_{\pm \ii\kappa } f$ of $f$ onto $E_{\pm\ii\kappa }$ vanishes. In this case it has a unique solution
in the orthogonal complement of $E_{\pm \ii\kappa }$ in $X$ which depends continuously upon $f$.
Similarly, equation
\eqref{Solve it 2} is solvable if the orthogonal projection $\Pi_0f$ of $f$ onto $E_0$ vanishes (note that this is merely a sufficient condition), and in this case
has a unique solution in the orthogonal complement of $E_0$ in $X$ which depends continuously upon $f$.
In the following analysis we use the convention that
$f_1=f_2=0$ and hence $P_0=\Pi_0=0$ if $0$ is not an eigenvalue of $L$.

\begin{proposition} \label{First iso}
The linear operator
$$
(I-\widetilde{\Pi}_{{\mathcal W}_1})\mathrm{d_1}F[0,0,0]\colon ({\mathcal W}_2\cap {\mathcal X})\rightarrow {\mathcal W}_2
$$
is an isomorphism.
\end{proposition}
\begin{proof}
The equation
\begin{equation}
(I-\widetilde{\Pi}_{{\mathcal W}_1})\mathrm{d}_1F[0,0,0](v)=w \label{First LS equation}
\end{equation}
with $w \in {\mathcal W}_2$ is equivalent to
$$
(\ii \kappa  kI-L)[v]_k=J^0(0)^{-1}[w]_k, \qquad k\in \mathbb{Z} \setminus \{0\},
$$
with $[w]_k \in Z$, $k \not\in \{0,-1,1\}$ and $[w]_1 \in E_{\ii \kappa }^\perp$, $[w]_{-1} \in E_{-\ii \kappa }^\perp$ (in $X$).
By assumption (H5) the operator $\ii \kappa  kI-L\colon X\rightarrow Z$ is an isomorphism for $k\notin \{0,-1,1\}$
and we have established that the equations
\begin{align*}
(\ii \kappa I-L)[v]_1&=J^0(0)^{-1}[w]_1,\\
(-\ii \kappa I-L)[v]_{-1}&=J^0(0)^{-1}[w]_{-1}
\end{align*}
have unique solutions $[v]_1 \in E_{\ii \kappa }^\perp$, $[v]_{-1} \in E_{-\ii \kappa }^\perp$ (in $X$) which depend continuously upon $[w]_1$, $[w]_{-1}$.
It follows that
\begin{align*}
\norm{v}_{L^2_\mathrm{per}(\mathbb{R},X)}^2&=\sum_{k\in\mathbb{Z}\backslash \{0,-1,1\}}\norm{[v]_k}_X^2+\|[v]_1\|_X^2+\|[v]_{-1}\|_X^2\\
&=\sum_{k\in \mathbb{Z}\backslash\{0,-1,1\}}\norm{(\ii \kappa kI-L)^{-1}J^0(0)^{-1}[w]_k}_X^2+\|[v]_1\|_X^2+\|[v]_{-1}\|_X^2\\
&\lesssim \sum_{k\in \mathbb{Z}\backslash\{0,-1,1\}}\norm{[w]_k}_Z^2+\|[w]_1\|_Z^2+\|[w]_{-1}\|_Z^2\\
&\leq \norm{w}_{{\mathcal Z}}^2
\end{align*}
and similarly
\[
\norm{v}_{H_\mathrm{per}^1(\mathbb{R},Z)}\lesssim \norm{w}_{{\mathcal Z}}
\]
(by assumption (H6)), so that $v$ lies in ${\mathcal X}$. We conclude that equation \eqref{First LS equation} has a unique solution $v \in {\mathcal W}_2 \cap {\mathcal X}$
which depends continuously upon $w \in {\mathcal W}_2$.
\end{proof}

\begin{lemma}
There exist neighbourhoods ${\mathcal U}_1$ and ${\mathcal U}_2$ of the origin in respectively ${\mathcal W}_1$ and ${\mathcal W}_2$ and
a reduction function $u_{{\mathcal W}_2}: {\mathcal U}_1 \times \Lambda_1 \times \Lambda_2 \rightarrow {\mathcal U}_2$ such that
equation \eqref{range-eq} admits the unique solution $(u_{{\mathcal W}_1},u_{{\mathcal W}_2}(u_{{\mathcal W}_1},\mu_1,\mu_2))$ 
in ${\mathcal U}_1 \times {\mathcal U}_2$. Furthermore $u_{{\mathcal W}_1}(0,0,0)=0$ and
$\mathrm{d}u_{{\mathcal W}_1}[0,0,0]=0$.
\end{lemma}
\begin{proof}
This result follows from the implicit-function theorem and Proposition \ref{First iso}.
\end{proof}

The next step is to further decompose the reduced equation
\begin{equation}
\widetilde{\Pi}_{{\mathcal W}_1} F(u_{{\mathcal W}_1}+u_{{\mathcal W}_2}(u_{{\mathcal W}_1},\mu_1,\mu_2),\mu_1,\mu_2)=0
\label{Red step 1}
\end{equation}
by introducing the orthogonal projection $\widetilde{\Pi}$ of ${\mathcal Z}$ onto
\begin{equation*}
{\mathcal W}_{1,1} =\{q f_1 + p f_2
+Ae_1\ee^{\ii \tau}+Be_2\ee^{\ii \tau}+\bar{A}\bar{e}_1\ee^{-\ii \tau}+\bar{B}\bar{e}_2\ee^{-\ii \tau},\ q,p \in {\mathbb R},\ A,B\in \mathbb{C}\} ,
\end{equation*}
which is given by
\begin{equation*}
\widetilde{\Pi}  u=\Pi_0[u]_0+\ee^{\ii \tau}\Pi_{\ii \kappa }[u]_1+\ee^{-\ii \tau}\Pi_{-\ii \kappa }[u]_{-1},
\end{equation*}
so that
$${\mathcal W}_1={\mathcal W}_{1,1} \oplus {\mathcal W}_{1,2},$$
where ${\mathcal W}_{1,2}=(I-\widetilde{\Pi}){\mathcal W}_1$.
Writing $u_{{\mathcal W}_1} \in {\mathcal U}_1$ as
\[
u_{{\mathcal W}_1}=\underbrace{\widetilde{\Pi} u_{{\mathcal W}_1}}_{\displaystyle =:u_1} + \underbrace{(I-\widetilde{\Pi} )u_{{\mathcal W}_1}}_{\displaystyle =:u_2},
\]
we find that \eqref{Red step 1} is equivalent to
\begin{align}
\widetilde{\Pi} G(u_1,u_2,\mu_1,\mu_2) &=0,\label{PP-eq} \\
\underbrace{(I-\widetilde{\Pi})G(u_1,u_2,\mu_1,\mu_2)}_{\displaystyle  = (I-\Pi_0)[G(u_1,u_2,\mu_1,\mu_2)]_0} \hspace{-5mm}&=0, \label{QQ-eq}
\end{align}
where
$$G(u_1,u_2,\mu_1,\mu_2)=F(u_1+u_2+u_{{\mathcal W}_2}(u_1+u_2,\mu_1,\mu_2),\mu_1,\mu_2).$$
We proceed with a further reduction of Lyapunov-Schmidt type. Solving equation \eqref{QQ-eq} for $u_2$ in terms of $u_1$, $\mu_1$ and $\mu_2$
requires hypothesis (H8) if the origin lies in the continuous spectrum of $L$ or is an eigenvalue embedded in the continuous spectrum.

\begin{lemma}

There exist neighbourhoods ${\mathcal U}_{1,1}$ and ${\mathcal U}_{1,2}$ of the origin in respectively ${\mathcal W}_{1,1}$ and ${\mathcal W}_{1,2}$ and
a reduction function $u_2: {\mathcal U}_{1,1} \times \Lambda_1 \times \Lambda_2 \rightarrow {\mathcal U}_{1,2}$ such that
equation \eqref{QQ-eq} admits the unique solution $(u_1,u_2(u_1,\mu_1,\mu_2))$ 
in ${\mathcal U}_{1,1} \times {\mathcal U}_{1,2}$. Furthermore $u_1(0,0,0)=0$ and
$\mathrm{d}u_1[0,0,0]=0$.
\end{lemma}
\begin{proof}
%
Equation \eqref{QQ-eq} is equivalent to
\begin{equation*}
Lu_2=J^0(0)^{-1}(I-\Pi_0)N^\star(u_1+u_2+u_{{\mathcal W}_2}(u_1+u_2,\mu_1,\mu_2),\mu_1,\mu_2),
\end{equation*}
Let $v(u_1,u_2,\mu_1,\mu_2)$ be the unique solution of the equation
$$Lv=J^0(0)^{-1}(I-\Pi_0)N^\star(u_1+u_2+u_{{\mathcal W}_2}(u_1+u_2,\mu_1,\mu_2),\mu_1,\mu_2)$$
and define $\Upsilon: ({\mathcal U}_1 \cap {\mathcal W}_{1,1}) \times
({\mathcal U}_1 \cap {\mathcal W}_{1,2}) \times \Lambda_1 \times \Lambda_2 \rightarrow
{\mathcal W}_{1,2}$ by
\[
\Upsilon(u_1,u_2,\mu_1,\mu_2)=u_2-v(u_1,u_2,\mu_1,\mu_2).
\]
Observing that
$\Upsilon(0,0,0,0)=0$, $\mathrm{d}_2\Upsilon[0,0,0,0]=I$, one therefore obtains the result from
the implicit-function theorem.
\end{proof}

The reduced equation
$$\widetilde{\Pi} G(u_1,u_2(u_1,\mu_1,\mu_2),\mu_1,\mu_2)=0$$
is conveniently written as
\begin{equation}\label{reduced-eq}
f(u_1 ,\mu_1,\mu_2)=0,
\end{equation}
where
\begin{equation*}
f(u_1 ,\mu_1,\mu_2)=\widetilde{\Pi} F(u_1 +h(u_1 ,\mu_1,\mu_2),\mu_1,\mu_2)
\end{equation*}
and the new reduction function
$h: {\mathcal U}_{1,1} \times \Lambda_1 \times \Lambda_2 
\rightarrow (I-\widetilde{\Pi}){\mathcal X}$ is given by
\begin{equation*}
h(u_1,\mu_1,\mu_2)=u_2(u_1 ,\mu_1,\mu_2)+u_{{\mathcal W}_2}(u_1+u_2(u_1 ,\mu_1,\mu_2),\mu_1,\mu_2).
\end{equation*}
Note again that $h(0,0,0)=0$ and $\mathrm{d}_1h[0,0,0]=0$.

Equation \eqref{reduced-eq} inherits the variational structure of \eqref{F eqn}.

\begin{proposition}
Equation \eqref{reduced-eq} is the Euler-Lagrange equation for the reduced action functional $s: {\mathcal U}_{1,1} \times \Lambda_1 \times \Lambda_2 \rightarrow
{\mathbb R}$ given by
\begin{align*}
s(u_1 ,\mu_1,\mu_2)&=S(u_1 +h(u_1,\mu_1,\mu_2),\mu_1,\mu_2),
\end{align*}
that is
\begin{equation}\label{hamiltonian_vecfield}
\mathrm{d}_1s[u_1 ,\mu_1,\mu_2](v_1 )=(f(u_1 ,\mu_1,\mu_2),v_1)
\end{equation}
for all $v_1 \in \widetilde{\Pi} {\mathcal X}$.
\end{proposition}
\begin{proof}
This result follows from the calculation
\begin{align*}
\mathrm{d}_1s[u_1 ,\mu_1,\mu_2](v_1 )
&=\mathrm{d}_1S[u_1 +h(u_1 ,\mu_1,\mu_2),\mu_1,\mu_2](v_1+\mathrm{d}_1h[u_1 ,\mu_1,\mu_2](v_1 ))\\
&=(F(u_1 +h(u_1 ,\mu_1,\mu_2),\mu_1,\mu_2),\mu_1,\mu_2),\mathrm{d}_1h[u_1 ,\mu_1,\mu_2](v_1 )+v_1 )\\
&=(\widetilde{\Pi}F(u_1 +h(u_1 ,\mu_1,\mu_2),\mu_1,\mu_2),\mathrm{d}_1h[u_1 ,\mu_1,\mu_2](v_1 )+v_1 )\\
&=(\widetilde{\Pi}F(u_1 +h(u_1 ,\mu_1,\mu_2),\mu_1,\mu_2),v_1 )\\
&=(f(u_1,\mu_1,\mu_2),v_1),
\end{align*}
where the second line follows from the first by Proposition \ref{Euler-Lagrange_A}, the third follows from the second because
$$(I-\widetilde{\Pi})F(u_1 +h(u_1 ,\mu_1,\mu_2),\mu_1,\mu_2)=0$$
by construction, and the fourth follows from the third because $h(u_1,\mu_1,\mu_2)$ (and hence all its derivatives) lies in $(I-\widetilde{\Pi})\tilde{X}$.
\end{proof}

Introducing coordinates
$$u_1 = q f_1 + p f_2
+Ae_1\ee^{\ii \tau}+Be_2\ee^{\ii \tau}+\bar{A}\bar{e}_1\ee^{-\ii \tau}+\bar{B}\bar{e}_2\ee^{-\ii \tau},$$
one finds that the reduced equation \eqref{reduced-eq} is given by
\begin{align}
\partial_{\bar{A}}s&=0, \label{Red 1} \\
\partial_{\bar{B}}s&=0, \label{Red 2} \\
\partial_ps&=0 \label{Red 3}
\end{align}
(recall that $s$ does not depend upon $q$ by hypothesis (H7)). The reduced action functional $s$ remains invariant under
the symmetries $T_\theta$ and $T$, whose actions on ${\mathcal W}_{1,1}$ are given by
\begin{align*}
T_\theta(A,B,\bar{A},\bar{B},q,p)&=(A\ee^{\ii \theta},B\ee^{\ii \theta},
\bar{A}\ee^{-\ii \theta},\bar{B}\ee^{-\ii \theta},q,p), \\
T(A,B,\bar{A},\bar{B},q,p)&=(\bar{A},\bar{B},A,B,-q,p).
\end{align*}
It follows that $s$ is a real-valued function of the real quantities $\abs{A}^2$, $\abs{B}^2$,
$\frac{\ii}{2}(\bar{A}B-A\bar{B})$,\linebreak
$\frac{1}{2}(A\bar{B}+\bar{A}B)$, $p$, $\mu_1$ and $\mu_2$ which is even with respect to
$\frac{\ii}{2}(\bar{A}B-A\bar{B})$. Restricting to $A=r_1$, $B=\ii r_2$, where $r_1$ and $r_2$ are real
(so that $\frac{\ii}{2}(\bar{A}B-A\bar{B})=r_1r_2$, $\frac{1}{2}(A\bar{B}+\bar{A}B)=0$), we find that
$$s(A,B,\bar{A},\bar{B},p,\mu_1,\mu_2)=\tilde{s}(r_1^2,r_2^2,r_1r_2,p,\mu_1,\mu_2),$$
where the right-hand side is even in its third argument, so that in fact, with a slight abuse of notation,
$$s(A,B,\bar{A},\bar{B},q,\mu_1,\mu_2)=\tilde{s}(r_1^2,r_2^2,p,\mu_1,\mu_2).$$
Equations \eqref{Red 1}--\eqref{Red 3} therefore reduce to
\begin{align*}
r_1\partial_1\tilde{s}(r_1^2,r_2^2,p,\mu_1,\mu_2) &=0,\\
r_2\partial_2\tilde{s}(r_1^2,r_2^2,p,\mu_1,\mu_2) &=0,\\
\partial_3 \tilde{s}(r_1^2,r_2^2,p,\mu_1,\mu_2)&=0,
\end{align*}
and further to
\begin{align}
\partial_1\tilde{s}(r_1^2,r_2^2,p,\mu_1,\mu_2) &=0, \label{Red 4} \\
\partial_2\tilde{s}(r_1^2,r_2^2,p,\mu_1,\mu_2)&=0, \label{Red 5} \\
\partial_3\tilde{s}(r_1^2,r_2^2,p,\mu_1,\mu_2) &=0 \label{Red 6}
\end{align}
for solutions with non-zero $r_1$ and $r_2$ components.

\begin{lemma} \label{For final IFT 1}
The quadratic parts of $\tilde{s}$ which are respectively independent of $(\mu_1,\mu_2)$,
independent of $\mu_2$ and linear in $\mu_1$, and independent of $\mu_1$ and linear in $\mu_2$ are
given by
\begin{align*}
\tilde{s}_2^{00} & =\tilde{s}_{002}^{00}p^2, \\
\tilde{s}_2^{10} & = \tilde{s}_{200}^{10} \mu_1 r_1^2 +\tilde{s}_{020}^{10} \mu_1 r_2^2+\tilde{s}_{002}^{10}\mu_1 p^2, \\
\tilde{s}_2^{01} & = -s_1 \mu_2 r_1^2 -s_2 \mu_2 r_2^2,
\end{align*}
where
\begin{align*}
\tilde{s}_{002}^{00}&= -\tfrac{1}{2},\\
\tilde{s}_{200}^{10}&=\Omega_0^0(\partial_{\mu_1} L^0 e_1, \bar{e}_1), \\
\tilde{s}_{020}^{10}&= \Omega_0^0(\partial_{\mu_1} L^0 e_2,\bar{e}_2), \\
\tilde{s}_{002}^{10}&= \tfrac{1}{2}\Omega_0^0(\partial_{\mu_1} L^0f_2,f_2).
\end{align*}
\end{lemma}

\begin{proof} We begin by recording the formulae
\begin{align*}
\mathrm{d}_1^2S[0,0,0](v_1,v_2)&=(J^0(0)(\kappa  v_{1\tau}- Lv_1),v_2), \\
\mathrm{d}_1^2\mathrm{d}_2S[0,0,0](v_1,v_2,1)&=(\partial_{\mu_1}J^0(0)(\kappa  v_{1\tau}- Lv_1)+J^0(0)\partial_{\mu_1} L^0 v_1,v_2),\\
\mathrm{d}_1^2\mathrm{d}_3S[0,0,0](v_1,v_2,1)&=(J^0(0)v_{1\tau},v_2)
\end{align*}
for $v_1$, $v_2\in{\mathcal X}$, which are obtained by differentiating the identity
$$
\mathrm{d}_1S[u,\mu_1,\mu_2](v)=\big(J^{\mu_1}(u)\big((\kappa +\mu_2)u_\tau-L^{\mu_1}u-N^{\mu_1}(u)\big),v\big)
$$
for $(u,\mu_1,\mu_2)\in {\mathcal U}\times\Lambda_1\times\Lambda_2$ and $v\in {\mathcal X}$
(see Proposition \ref{Euler-Lagrange_A}).

These formulae show that
$$
\tilde{s}_2^{00} = \tilde{s}_{200}^{00}r_1^2 + \tilde{s}_{020}^{00}r_2^2 + \tilde{s}_{002}^{00}p^2,
$$
where
\begin{align*}
\tilde{s}_{200}^{00}&=\mathrm{d}_1^2 S[0,0,0](\ee^{\ii \tau}e_1,\ee^{-\ii \tau}\bar{e}_1)=(J^0(0)(\ii\kappa  I - L)\ee^{\ii \tau}e_1, \ee^{\ii \tau}e_1)=0, \\[0.5mm]
\tilde{s}_{020}^{00}&=\mathrm{d}_1^2 S[0,0,0](\ee^{\ii \tau}e_2,\ee^{-\ii \tau}\bar{e}_2)=(J^0(0)(\ii\kappa  I - L)\ee^{\ii \tau}e_2, \ee^{\ii \tau}e_2)=0, \\[0.5mm]
\tilde{s}_{002}^{00}&=\tfrac{1}{2}\mathrm{d}_1^2 S[0,0,0](f_2,f_2)=-\tfrac{1}{2}(J^0(0)Lf_2,f_2)=-\tfrac{1}{2}\Omega_0^0(f_1,f_2)=-\tfrac{1}{2}.
\end{align*}
Similarly,  denoting the part of $h(u_1,\mu_1,\mu_2)$
which is homogeneous of degree $i$, $j$, $k$, $\ell$, $n_1$ and $n_2$ in respectively
$A$, $B$, $\bar{A}$, $\bar{B}$, $\mu_1$ and $\mu_2$ by $h_{ijk\ell}^{n_1n_2}A^iB^j\bar{A}^k\bar{B}^\ell\mu_1^{n_1}\mu_2^{n_2}$,
we find that
\begin{align*}
\tilde{s}_2^{10} &= \tilde{s}_{200}^{10}\mu_1r_1^2 + \tilde{s}_{020}^{10}\mu_1r_2^2 + \tilde{s}_{002}^{10}\mu_1p^2, \\
\tilde{s}_2^{01} &= \tilde{s}_{200}^{01}\mu_2r_1^2 + \tilde{s}_{020}^{01}\mu_2r_2^2 + \tilde{s}_{002}^{01}\mu_2p^2,
\end{align*}
where
\begin{align*}
\tilde{s}_{200}^{10}&=\mathrm{d}_1^2\mathrm{d}_2S[0,0,0](\ee^{\ii \tau}e_1,\ee^{-\ii\tau}\bar{e}_1,1)+\mathrm{d}_1^2 S[0,0,0](\ee^{\ii \tau}e_1,h_{00100}^{10})
+\mathrm{d}_1^2 S[0,0,0](\ee^{-\ii \tau}\bar{e}_1,h_{10000}^{10}),\\
& = (J^0(0)\partial_{\mu_1} L^0 (\ee^{\ii \tau}e_1),\ee^{\ii \tau}e_1) \\
& = \Omega_0^0(\partial_{\mu_1} L^0 (e_1),\bar{e}_1), \\[0.5mm]
\tilde{s}_{020}^{10}&=\mathrm{d}_1^2\mathrm{d}_2S[0,0,0](\ee^{\ii \tau}e_2,\ee^{-\ii\tau}\bar{e}_2,1)+\mathrm{d}_1^2 S[0,0,0](\ee^{\ii \tau}e_2,h_{00010}^{10})
+\mathrm{d}_1^2 S[0,0,0](\ee^{-\ii \tau}\bar{e}_1,h_{01000}^{10}) \\
& = (J^0(0)\partial_{\mu_1} L^0 (\ee^{\ii \tau}e_2),\ee^{\ii \tau}e_2) \\
& = \Omega_0^0(\partial_{\mu_1} L^0 (e_2),\bar{e}_2),\displaybreak \\[0.5mm]
\tilde{s}_{002}^{10}&=\tfrac{1}{2}\mathrm{d}_1^2\mathrm{d}_2S[0,0,0](f_2,f_2,1)+\mathrm{d}_1^2 S[0,0,0](f_2,h^{10}_{00001}) \\
& = \tfrac{1}{2}(J^0(0)\partial_{\mu_1} L^0f_2,f_2) \\
& = \tfrac{1}{2}\Omega_0^0(\partial_{\mu_1} L^0f_2,f_2)
\end{align*}
and
\begin{align*}
\tilde{s}_{200}^{01}&=\mathrm{d}_1^2\mathrm{d}_3S[0,0,0](\ee^{\ii \tau}e_1,\ee^{-\ii\tau}\bar{e}_1,1)+\mathrm{d}_1^2 S[0,0,0](\ee^{\ii \tau}e_1,h_{00100}^{01})
+\mathrm{d}_1^2 S[0,0,0](\ee^{-\ii \tau}\bar{e}_1,h_{10000}^{01}) \\
& = \ii (J^0(0)\ee^{\ii \tau}e_1,\ee^{\ii \tau}e_1) \\
& = \ii \Omega_0^0(e_1,\bar{e}_1) \\
& = - s_1, \\[0.5mm]
\tilde{s}_{020}^{01}&=\mathrm{d}_1^2\mathrm{d}_3S[0,0,0](\ee^{\ii \tau}e_2,\ee^{-\ii\tau}\bar{e}_2,1)+\mathrm{d}_1^2 S[0,0,0](\ee^{\ii \tau}e_2,h_{00100}^{01})
+\mathrm{d}_1^2 S[0,0,0](\ee^{-\ii \tau}\bar{e}_2,h_{01000}^{01}) \\
& = \ii (J^0(0)\ee^{\ii \tau}e_2,\ee^{\ii \tau}e_2) \\
& = \ii \Omega_0^0(e_2,\bar{e}_2) \\
& = - s_2, \\[0.5mm]
\tilde{s}_{002}^{01}&=\tfrac{1}{2}\mathrm{d}_1^2\mathrm{d}_3S[0,0,0](f_2,f_2,1)+\mathrm{d}_1^2 S[0,0,0](f_2,h^{10}_{00001})\\
& = 0.
\end{align*}
Note that the second derivatives are extended \emph{bilinearly} to the complexification of ${\mathcal Z}$ while $(\cdot\,,\cdot)$
is extended \emph{sesquilinearly}.
\end{proof}

\begin{corollary} \label{For final IFT 2}
One has the formulae
$$\tilde{s}_{200}^{10} = \partial_{\mu_1} \kappa_1^0, \qquad \tilde{s}_{200}^{01} = \partial_{\mu_1} \kappa_2^0.$$
\end{corollary}
\begin{proof}
Observe that
$$\Omega_0^{\mu_1}(L^{\mu_1} e_i^{\mu_1},\bar{e}_i^{\mu_1}) = \ii \kappa_i^{\mu_1} \Omega_0^{\mu_1}(e_i^{\mu_1},\bar{e}_i^{\mu_1}) = - s_i \kappa_i^{\mu_1}.$$
Differentiating this formula with respect to $\mu_1$ and evaluating the result at $\mu_1=0$ yields
\begin{align*}
-s_i \partial_{\mu_1} \kappa_i^0\big|_{\mu_1=0} & = \Omega_0^0(\partial_{\mu_1} L^0 e_i,\bar{e}_i)+
\Omega_0^1(L e_i, \bar{e}_i) + \Omega_0^0(L \partial_{\mu_1} e_i^{\mu_1},\bar{e}_i) + \Omega_0^0( Le_i, \partial_{\mu_1}\bar{e}_i^{\mu_1})\Big|_{\mu_1=0} \\
& = \Omega_0^0(\partial_{\mu_1} L^0e_i, \bar{e}_i) + \Omega_0^1(L e_i, \bar{e}_i) 
- \Omega_0^0(\partial_{\mu_1} e_i^{\mu_1},L\bar{e}_i) + \Omega_0^0( Le_i, \partial_{\mu_1}\bar{e}_i^{\mu_1})\Big|_{\mu_1=0} \\
& = \Omega_0^0(\partial_{\mu_1} L^0 e_i, \bar{e}_i) +\ii \kappa\big(\Omega_0^1(e_i, \bar{e}_i) 
+ \Omega_0^0(\partial_{\mu_1} e_i^{\mu_1},\bar{e}_i) + \Omega_0^0( e_i, \partial_{\mu_1}\bar{e}_i^{\mu_1})\big)\Big|_{\mu_1=0} \\ 
& = \Omega_0^0 (\partial_{\mu_1} L^0e_i,\bar{e}_i)+\ii\kappa \partial_{\mu_1} \underbrace{\Omega_0^{\mu_1} (e_i^{\mu_1},\partial_{\mu_1}\bar{e}_i^{\mu_1})}_{\displaystyle = s_i} \Big|_{\mu_1=0}  \\
& = \Omega_0^0 (\partial_{\mu_1} L^0e_i,\bar{e}_i).\qedhere
\end{align*}
\end{proof}

Finally, we solve equations \eqref{Red 4}--\eqref{Red 6} using the information given by Lemma \ref{For final IFT 1} and Corollary \ref{For final IFT 2},
thus completing the proof of Theorem \ref{LCT}.

\begin{lemma}\label{main_result_lemma}
There exist $\varepsilon>0$ and functions $p^\star: B_\varepsilon(0) \rightarrow {\mathbb R}$,
$\mu_1^\star: B_\varepsilon(0) \rightarrow {\mathbb R}$,
$\mu_2^\star: B_\varepsilon(0) \rightarrow {\mathbb R}$
such that the solution set of \eqref{Red 4}--\eqref{Red 6} in ${\mathcal U}_{1,1} \times \Lambda_1 \times \Lambda_2$
coincides with
$$\{(r_1^2,r_2^2,p^\star(r_1^2,r_2^2),\mu_1^\star(r_1^2,r_2^2),\mu_2^\star(r_1^2,r_2^2): |(r_1^2,r_2^2)| < \varepsilon\}.$$
\end{lemma}
\begin{proof}
Since
$$\begin{pmatrix}\partial_1 \tilde{s}(0,0,0,0,0) \\ \partial_2 \tilde{s}(0,0,0,0,0) \end{pmatrix}=\begin{pmatrix} 0 \\ 0 \end{pmatrix}$$
and
\begin{align*}
\det\begin{pmatrix}
\partial_1\partial_4 \tilde{s}(0,0,0,0,0) & \partial_1 \partial_5 \tilde{s}(0,0,0,0,0)\\
\partial_2 \partial_4 \tilde{s}(0,0,0,0,0) & \partial_2\partial_5\tilde{s}(0,0,0,0,0)
\end{pmatrix}
&=-s_2\tilde{s}_{200}^{10}+s_1\tilde{s}_{020}^{01} \\
&= -s_1s_2\partial_{\mu_1}(\kappa_1^0-\kappa_2^0)\\
& \neq 0,
\end{align*}
we can solve equations \eqref{Red 4}, \eqref{Red 5} locally for $\mu_1=\mu_1(r_1^2,r_2^2,p)$,
$\mu_2=\mu_2(r_1^2,r_2^2,p)$ using the implicit-function theorem. Inserting this solution into
\eqref{Red 6} yields
\begin{equation}
\tilde{t}(r_1^2,r_2^2,p)=0, \label{Red 7}
\end{equation}
where
$$
\tilde{t}(r_1^2,r_2^2,p)=\partial_3\tilde{s}(r_1^2,r_2^2,p,\mu_1(r_1^2,r_2^2,p),\mu_2(r_1^2,r_2^2,p)).
$$
Furthermore
$$\tilde{t}(0,0,0)=\partial_3\tilde{s}(0,0,0,\mu_1(0,0,0),\mu_2(0,0,0))=0$$
and similarly
\begin{align*}
\partial_3 \tilde{t}(0,0,0) &= 
\partial_3^2\tilde{s}(0,0,0,0,0)+\partial_3\mu_1(0,0,0)\partial_3\partial_4\tilde{s}(0,0,0,0,0)
+\partial_3\mu_2(0,0,0)\partial_3\partial_5\tilde{s}(0,0,0,0,0) \\
& = 2\tilde{s}_{002}^{00} \\
& = -1.
\end{align*}
We can therefore solve equation \eqref{Red 6} locally for $p=p(r_1^2,r_2^2)$ using the implicit-function theorem.

The assertion follows by setting $p^\star(r_1^2,r_2^2)=p(r_1^2,r_2^2)$,
$\mu_1^\star(r_1^2,r_2^2)=\mu_1(r_1^2,r_2^2,p(r_1^2,r_2^2))$ and
$\mu_2^\star(r_1^2,r_2^2)=\mu_2(r_1^2,r_2^2,p(r_1^2,r_2^2))$.
\end{proof}

\section{Hydroelastic waves} \label{Formulation}

In this section we introduce the hydrodynamic problem for travelling waves on the surface of a three-dimensional body of water beneath a thin ice sheet modelled using the Cosserat theory of
hyperelastic shells (Plotnikov \& Toland \cite{PlotnikovToland11}). The fluid is bounded below by a rigid horizontal bottom $\{{x_2}=-h\}$ (the cases $h<\infty$
and $h=\infty$ are referred to as `finite depth' and `infinite depth') and
above by a free surface $\{{x_2}=\eta({x_1},{x_3})\}$ (in a frame of reference following the wave with constant speed $c$ in the ${x_1}$ direction); there is no cavitation between this surface
and the ice sheet. Working in a dimensionless coordinates with unit length $(D/\rho g)^{1/4}$ and unit speed $(\rho/ Dg^3)^{1/8}$, one finds that the hydrodynamic problem
is to find an Eulerian velocity potential $\phi$ which satisfies the equations
\pagebreak
\begin{alignat}{2}
&\phi_{{x_1}{x_1}}+\phi_{{x_2}{x_2}}+\phi_{{x_3}{x_3}}=0,&& -\tfrac{1}{\beta}<{x_2}<\eta({x_1},{x_3}), \label{gov 1}\\
&\phi_{x_2}=0, && {x_2}=-\tfrac{1}{\beta}, \label{gov 2}\\
&\phi_{x_2}+\gamma \eta_{x_1}-\phi_{x_1}\eta_{x_1}-\phi_{x_3}\eta_{x_3}=0,&& {x_2}=\eta({x_1},{x_3}), \label{gov 3}\\
&-\gamma\phi_{x_1}+\tfrac{1}{2}(\phi_{x_1}^2+\phi_{x_2}^2+\phi_{x_3}^2)+\eta+U(\eta)=0,\qquad &&{x_2}=\eta({x_1},{x_3}), \label{gov 4}
\end{alignat}
where
\begin{align*}
U(\eta)&=2\bigg(\!\frac{1}{\sqrt{Q(\eta)}}\bigg[\partial_{x_1}\bigg(\frac{1+\eta_{x_3}^2}{\sqrt{Q(\eta)}}P(\eta)_{x_1}\bigg)-\partial_{x_1}\bigg(\frac{\eta_{x_1}\eta_{x_3}}{\sqrt{Q(\eta)}}P(\eta)_{x_3}\bigg)\\
& \hspace{1.05in}
\mbox{}-\partial_{x_3}\bigg(\frac{\eta_{x_1}\eta_{x_3}}{\sqrt{Q(\eta)}}P(\eta)_{x_1}\bigg)+\partial_{x_3}\bigg(\frac{1+\eta_{x_1}^2}{\sqrt{Q(\eta)}}P(\eta)_{x_3}\bigg)\bigg]+2P(\eta)^3-2K(\eta)P(\eta)\!\bigg),\\
Q(\eta)&=1+\eta_{x_1}^2+\eta_{x_3}^2,\\
P(\eta)&=\frac{1}{2Q(\eta)^{3/2}}\bigg[(1+\eta_{x_3}^2)\eta_{{x_1}{x_1}}-2\eta_{{x_1}{x_3}}\eta_{x_1}\eta_{x_3}+(1+\eta_{x_1}^2)\eta_{{x_3}{x_3}}\bigg],\\
K(\eta)&=\frac{1}{Q(\eta)^2}(\eta_{{x_1}{x_1}}\eta_{{x_3}{x_3}}-\eta_{{x_1}{x_3}}^2),
\end{align*}
and
$$\beta=\left(\frac{D}{\rho g h^4}\right)^{\!\!1/4} \geq 0, \qquad \gamma=\left(\frac{c^8 \rho}{D g^3}\right)^{\!\!1/8}>0,$$
where $D$, $\rho$ and $g$
are respectively the coefficient of flexural rigidity for the ice sheet, the density of the fluid and the acceleration due to gravity
(see Guyenne and Parau \cite{GuyenneParau12}).

We consider waves which are periodic with periods $p_1$ and $p_2$ in two arbitrary horizontal directions $x$ and $z$
which form (different) angles $\theta_1$, $\theta_2 \in [0,\pi)$ with the ${x_1}$-axis respectively, so that
\[
x=\csc(\theta_2-\theta_1)({x_1}\sin\theta_2-{x_3}\cos\theta_2), \qquad
z=\csc(\theta_1-\theta_2)({x_1}\sin\theta_1-{x_3}\cos\theta_1)
\]
(see Figure \ref{intro - periodic domain}). 
To this end we seek solutions of the governing equations of the form
$$
\eta({x_1},{x_3})=\tilde{\eta}(\tilde{x},\tilde{z}),\qquad
\phi({x_1},{x_2},{x_3})=\tilde{\phi}(\tilde{x},{x_2},\tilde{z}),
$$
where
\begin{equation*}
\tilde{x} = {x_1}\sin\theta_2-{x_3}\cos\theta_2, \qquad
\tilde{z} = \frac{2\pi}{p_2}({x_1}\sin\theta_1-{x_3}\cos\theta_1)
\end{equation*}
and $\tilde{\eta}$, $\tilde{\phi}$ are $2\pi$-periodic in $\tilde{z}$ (the requirement that they are
also periodic in $\tilde{x}$ is applied later). 
The governing equations become
\pagebreak
\begin{alignat}{2}
&\phi_{xx}+\phi_{{x_2}{x_2}}+\nu^2\phi_{zz}+2\nu\cos(\theta_1-\theta_2)\phi_{xz}=0, &&-\tfrac{1}{\beta}<{x_2}<\eta, \label{basic 1}\\
&\phi_{x_2}=0, &&{x_2}=-\tfrac{1}{\beta}, \label{basic 2} \\
&\phi_{{x_2}}=-\gamma(\sin \theta_2\eta_x+\nu\sin \theta_1\eta_z)+\eta_x\phi_x+\nu^2\eta_z\phi_z \nonumber\\
& \qquad\qquad\mbox{}+\nu\cos(\theta_1-\theta_2)(\eta_x\phi_z+\eta_z\phi_x),\  &&{x_2}=\eta, \label{basic 3} \\
&-\gamma(\sin \theta_2\phi_x+\nu\sin \theta_1\phi_z) \nonumber \\
& \qquad\mbox{}+\tfrac{1}{2}(\phi_x^2+\phi_{x_2}^2+\nu^2\phi_z^2+2\nu\cos(\theta_1-\theta_2)\phi_x\phi_z)+\gamma\eta+U(\eta)=0,\quad\  &&{x_2}=\eta, \label{basic 4}
\end{alignat}
where
\begin{align*}
U(\eta)&=2\bigg(\frac{1}{\sqrt{Q(\eta)}}\bigg[\partial_x\bigg(\frac{1+\nu^2\sin^2(\theta_1-\theta_2)\eta_z^2}{\sqrt{Q(\eta)}}P(\eta)_x\bigg)+\nu^2\partial_z\bigg(\frac{1+\sin^2(\theta_1-\theta_2)\eta_x^2}{\sqrt{Q(\eta)}}P(\eta)_z\bigg)\\
& \hspace{1.05in}
\mbox{}+\nu\cos(\theta_1-\theta_2)\bigg(\partial_x\bigg(\frac{P(\eta)_z}{\sqrt{Q(\eta)}}\bigg)+\partial_z\bigg(\frac{P(\eta)_x}{\sqrt{Q}}\bigg)\bigg) \\
& \hspace{1.05in}
-\nu^2\sin^2(\theta_1-\theta_2)\bigg(\partial_x\bigg(\frac{\eta_x\eta_z}{\sqrt{Q(\eta)}}P(\eta)_z\bigg)+\partial_z\bigg(\frac{\eta_x\eta_z}{\sqrt{Q(\eta)}}P(\eta)_x\bigg)\bigg]\\
& \hspace{1.05in}
\mbox{}+2P(\eta)^3-2K(\eta)P(\eta)\bigg),\\
Q(\eta)&=1+\eta_x^2+\nu^2\eta_z^2+2\nu\cos(\theta_1-\theta_2)\eta_x\eta_z,\\
P(\eta)&=\frac{1}{2Q(\eta)^\frac{3}{2}}\bigg[\eta_{xx}+\nu^2\eta_{zz}+2\nu\cos(\theta_1-\theta_2)\eta_{xz} +\nu^2\sin^2(\theta_1-\theta_2)(\eta_{xx}\eta_z^2-2\eta_{xz}\eta_x\eta_z+\eta_{zz}\eta_x^2)\bigg],\\
K(\eta)&=\frac{1}{Q(\eta)^2}\nu^2\sin^2(\theta_1-\theta_2)(\eta_{xx}\eta_{zz}-\eta_{xz}^2),
\end{align*}
the tildes have been dropped for notational simplicity, and $\nu=2\pi/p_2$.


We proceed by formulating equations \eqref{basic 1}--\eqref{basic 4} as a Hamiltonian system in which the horizontal spatial direction $x$ plays the role of
the time-like variable (`spatial dynamics'). Our starting point is the observation that these equations follow from the formal variational principle
\begin{align}
\delta \int \int_{0}^{2\pi} \Bigg\{ & \int_{-\frac{1}{\beta}}^{\eta} \frac{1}{2} \left( \phi_{x}^{2}+\phi_{x_2}^{2}+\nu^{2} \phi_{z}^{2}+2\nu \cos(\theta_1-\theta_2)\phi_{x}\phi_{z}\right)\dx_2\nonumber \\
& \mbox{}+2\sqrt{Q(\eta)}P(\eta)^2
+\tfrac{1}{2} \eta^{2}  +\gamma(\eta_x \sin\theta_2+\nu\eta_z \sin\theta_1)\phi|_{{x_2}=\eta} \Bigg\} \dz\dx=0, \label{Luke's VP}
\end{align}
in which the variations are taken over \(\eta\) and \(\phi\) (a modified version of the classical variational principle introduced by Luke \cite{Luke67}). Because of the difficulty
in performing analysis on a variable domain, we use the change of variable
$$
\phi(x,{x_2},z)=\Phi(x,y,z), \qquad {x_2}=\begin{cases} y+(1+\beta y)\eta, & \beta>0, \\[2mm]
y+\ee^y \eta, & \beta=0, \end{cases}
$$
to map the variable fluid domain $\{-\frac{1}{\beta} < {x_2} < \eta(x,z)\}$ to the fixed domain $\{-\frac{1}{\beta}<y<0\}$. The variational principle
\eqref{Luke's VP} is transformed into
\begin{equation*}
\delta \mathcal{L}=0, \qquad \mathcal{L}=\int L(\eta,\eta_x,\eta_{xx},\Phi,\Phi_x) \dx,
\end{equation*}
in which
\begin{eqnarray*}
\lefteqn{L(\eta,\eta_x,\eta_{xx},\Phi,\Phi_x)}\\
& & =\int_0^{2\pi}\bigg\{\int_{-\frac{1}{\beta}}^0 \frac{1}{2K_2(\eta)}\bigg((\Phi_x-K_1(\eta)\eta_x\Phi_y)^2+K_2(\eta)^2\Phi_y^2+\nu^2(\Phi_z-K_1(\eta)\eta_z\Phi_y) \\
& & \hspace{4.4cm}\mbox{}+2\nu\cos(\theta_1-\theta_2)(\Phi_x-K_1(\eta)\eta_x\Phi_y)(\Phi_z-K_1(\eta)\eta_z\Phi_y) \bigg)\dy \\
& & \hspace{1.75cm}\mbox{}+2\sqrt{Q(\eta)}P(\eta)^2
+\tfrac{1}{2} \eta^{2}  +\gamma(\eta_x \sin\theta_2+\nu\eta_z \sin\theta_1)\Phi|_{y=0}\bigg\}\dz
\end{eqnarray*}
and
$$K_1(\eta)  = \begin{cases}
\dfrac{1+\beta y}{1+\beta \eta}, & \beta>0, \\[5mm]
\dfrac{\ee^y}{1+\ee^y\eta}, & \beta=0,
\end{cases}
\qquad K_2(\eta) = \begin{cases}
\dfrac{1}{1+\beta \eta}, & \beta>0, \\[5mm]
\dfrac{1}{1+\ee^y\eta}, & \beta=0.
\end{cases}$$

The next step is to perform a formal Legendre transformation (see Lanczos \cite[Appendix I]{Lanczos}) by
introducing the new coordinate
$$\rho = \eta_x$$
and momenta
\begin{align*}
\zeta&=\frac{\delta L}{\delta\eta_x}-\frac{\mathrm{d}}{\mathrm{d}x}\left(\frac{\delta L}{\delta \eta_{xx}}\right) \\
& =-\int_{-\frac{1}{\beta}}^0 \frac{K_1(\eta)}{K_2(\eta)}
\Big((\Phi_x-K_1(\eta)\eta_x\Phi_y)\Phi_y + \nu\cos(\theta_1-\theta_2)(\Phi_z-K_1(\eta)\eta_z\Phi_y)\Phi_y \Big)\dy\\
& \qquad\mbox{}-10\frac{P(\eta)^2}{\sqrt{Q(\eta)}}(\eta_x+\nu\cos(\theta_1-\theta_2)\eta_z)\\
& \qquad\mbox{}+\frac{4P(\eta)}{Q(\eta)}\nu^2\sin^2(\theta_1-\theta_2)(-\eta_{xz}\eta_z+\eta_{zz}\eta_x) + \gamma\sin \theta_2 \Phi|_{y=0}\\
& \qquad\mbox{}-\frac{\mathrm{d}}{\mathrm{d}x}\left(\frac{2P(\eta)}{Q(\eta)}(1+\nu^2\sin^2(\theta_1-\theta_2)\eta_z^2)\right),\\[2mm]
\xi&=\frac{\delta L}{\delta\eta_{xx}} \\
&=\frac{2P(\eta)}{Q(\eta)}(1+\nu^2\sin^2(\theta_1-\theta_2)\eta_z^2), \displaybreak[4]\\[2mm]
\Psi&=\frac{\delta L}{\delta \Phi_x} \\
&=\frac{1}{K_2(\eta)}(\Phi_x - K_1(\eta)\eta_x\Phi_y)+\frac{1}{K_2(\eta)}\nu\cos(\theta_1-\theta_2)(\Phi_z - K_1(\eta)\eta_z\Phi_y),
\end{align*}
and defining the Hamiltonian by
\begin{align}
H(&\eta,\rho,\Phi,\zeta,\xi,\Psi) \nonumber\\
&=\int_S\zeta\eta_x\dz+\int_S\xi\eta_{xx}\dz+\int_\Sigma \Psi\Phi_x\dy\dz-L(\eta,\eta_x,\eta_{xx},\Phi,\Phi_x) \nonumber\\
&=\int_\Sigma \bigg\{\frac{K_2(\eta)}{2}(\Psi^2-\Phi_y^2)
- \frac{1}{2K_2(\eta)}\nu^2\sin^2(\theta_1-\theta_2)(\Phi_z-K_1(\eta)\eta_z\Phi_y)^2 \nonumber \\
& \hspace{1.5cm}\mbox{}+K_1(\eta)\rho\Phi_y\Psi - \nu\cos(\theta_1-\theta_2)(\Phi_z-K_1(\eta)\eta_z\Phi_y)\Psi\bigg\}\dy\dz \nonumber \\
& \qquad\mbox{}+\int_S \bigg\{\zeta\rho
-\frac{1}{2}\eta^2-\gamma(\rho\sin\theta_2+\nu\eta_z\sin\theta_1)\Phi|_{y=0}
+\frac{Q(\eta,\rho)^{5/2}\xi^2}{2(1+\nu^2\sin^2(\theta_1-\theta_2)\eta_z^2)^2} \nonumber \\
& \hspace{2.1cm}\mbox{}+\frac{\xi}{1+\nu^2\sin^2(\theta_1-\theta_2)\eta_z^2}
\big(
\!-\!(1+\sin^2(\theta_1-\theta_2)\rho^2)\nu^2\eta_{zz} \nonumber \\
& \hspace{2.65in}\mbox{}+2\nu^2\sin^2(\theta_1-\theta_2)\rho\rho_z\eta_z-2\nu\cos(\theta_1-\theta_2)\rho_z\big)\!\bigg\}\!\dz, \label{Formula for H}
\end{align}
where $S=(0,2\pi)$, $\Sigma=(-\tfrac{1}{\beta},0)\times (0,2\pi)$ and
$$Q(\eta,\rho) = 1+\rho^2+\nu^2\eta_z^2+2\nu\cos(\theta_1-\theta_2)\rho\eta_z.$$
Hamilton's equations are
\begin{align}
\eta_x & = \frac{\delta H}{\delta \zeta} \nonumber \\
& =\rho, \label{Ham eq 1}\\[2mm]
\rho_x & = \frac{\delta H}{\delta \xi} \nonumber \\
& = \frac{Q(\eta,\rho)^{5/2}\xi}{(1+\nu^2\sin^2(\theta_1-\theta_2)\eta_z^2)^2}
-\frac{(1+\sin^2(\theta_1-\theta_2)\rho^2)\nu^2\eta_{zz}}{1+\nu^2\sin^2(\theta_1-\theta_2)\eta_z^2} \nonumber \\
& \qquad\mbox{}
-\frac{2\nu\cos(\theta_1-\theta_2)\rho_z}{1+\nu^2\sin^2(\theta_1-\theta_2)\eta_z^2} 
+\frac{2\nu^2\sin^2(\theta_1-\theta_2)\rho\rho_z\eta_z}{1+\nu^2\sin^2(\theta_1-\theta_2)\eta_z^2} , \label{Ham eq 2}\\[2mm]
\Phi_x & = \frac{\delta H}{\delta \Psi} \nonumber \\
& = K_2(\eta)\Psi +K_1(\eta)\rho\Phi_y -\nu\cos(\theta_1-\theta_2)(\Phi_z-K_1(\eta)\eta_z\Phi_y), \label{Ham eq 3}\displaybreak\\[2mm]
-\zeta_x & =  \frac{\delta H}{\delta \eta} \nonumber \\
& = \int_{-\frac{1}{\beta}}^0 \bigg\{-\frac{K_2(\eta)^2K_3}{2}(\Psi^2-\Phi_y^2)
- \frac{K_3}{2}\nu^2\sin^2(\theta_1-\theta_2)(\Phi_z-K_1(\eta)\eta_z\Phi_y)^2  \nonumber \\
& \hspace{2cm}\mbox{}-K_1(\eta)K_2(\eta)K_3\rho\Phi_y\Psi-\nu^2\sin^2(\theta_1-\theta_2)K_1(\eta)K_3\eta_z\Phi_y(\Phi_z-K_1(\eta)\eta_z\Phi_y)\nonumber \\
& \hspace{2cm}\mbox{}-\nu\cos(\theta_1-\theta_2)K_1(\eta)K_2(\eta)K_3\eta_z\Phi_y\Psi -[\nu\cos(\theta_1-\theta_2)K_1(\eta)\Phi_y\Psi]_z \nonumber \\
& \hspace{2cm}\mbox{}-\left[\frac{K_1(\eta)}{K_2(\eta)}\nu^2\sin^2(\theta_1-\theta_2)(\Phi_z-K_1(\eta)\eta_z\Phi_y)\Phi_y\right]_z
\bigg\}\dy\nonumber \\
& \qquad\mbox{}+\left[\frac{2\nu^2\sin^2(\theta_1-\theta_2)Q(\eta,\rho)^{5/2}\xi^2\eta_z}{(1+\nu^2\sin^2(\theta_1-\theta_2)\eta_z^2)^3}\right]_z \nonumber \\
& \qquad\mbox{}-\left[\frac{5\nu^2Q(\eta,\rho)^{3/2}\xi^2\eta_z}{2(1+\nu^2\sin^2(\theta_1-\theta_2)\eta_z^2)^2}\right]_z
-\left[\frac{5\nu\cos(\theta_1-\theta_2)Q(\eta,\rho)^{3/2}\xi^2\rho}{2(1+\nu^2\sin^2(\theta_1-\theta_2)\eta_z^2)^2}\right]_z\nonumber \\
& \qquad\mbox{}+\bigg[\frac{2\nu\sin^2(\theta_1-\theta_2)\xi\eta_z}{(1+\nu^2\sin^2(\theta_1-\theta_2)\eta_z^2)^2}
\big(
\!\!-\!(1+\sin^2(\theta_1-\theta_2)\rho^2)\nu^2\eta_{zz}\nonumber \\
& \hspace{2.5in}\mbox{}+2\nu^2\sin^2(\theta_1-\theta_2)\rho\rho_z\eta_z-2\nu\cos(\theta_1-\theta_2)\rho_z\big)\!\bigg]_z\nonumber \\
& \qquad\mbox{}-\left[\frac{(1+\sin^2(\theta_1-\theta_2)\rho^2)\nu^2\xi}{1+\nu^2\sin^2(\theta_1-\theta_2)\eta_z^2}\right]_{zz}
\!\!-\left[\frac{2\nu^2\sin^2(\theta_1-\theta_2)\rho\rho_z\xi}{1+\nu^2\sin^2(\theta_1-\theta_2)\eta_z^2}\right]_z
-\eta+\gamma\nu\sin\theta_1\Phi_z|_{y=0}, \label{Ham eq 4}\\[2mm]
-\xi_x & = \frac{\delta H}{\delta \rho} \nonumber \\
& =\zeta-\gamma\sin\theta_2\Phi|_{y=0} + \int_{-\frac{1}{\beta}}^0 K_1(\eta)\Psi\Phi_y \dy
+ \frac{5Q(\eta,\rho)^{3/2}\xi^2\rho}{2(1+\nu^2\sin^2(\theta_1-\theta_2)\eta_z^2)^2} \nonumber \\
& \qquad\mbox{}+ \frac{5\nu\cos(\theta_1-\theta_2)Q(\eta,\rho)^{3/2}\xi^2\eta_z}{2(1+\nu^2\sin^2(\theta_1-\theta_2)\eta_z^2)^2} 
- \frac{2\nu^2\sin^2(\theta_1-\theta_2)\xi\rho\eta_{zz}}{1+\nu^2\sin^2(\theta_1-\theta_2)\eta_z^2}  \nonumber \\
& \qquad\mbox{}-\left[\frac{2\nu^2\sin^2(\theta_1-\theta_2)\xi\eta_z}{1+\nu^2\sin^2(\theta_1-\theta_2)\eta_z^2}\right]_z\rho
+\left[\frac{2\nu\cos(\theta_1-\theta_2)\xi}{1+\nu^2\sin^2(\theta_1-\theta_2)\eta_z^2}\right]_z, \label{Ham eq 5}\\[2mm]
-\Psi_x & =  \frac{\delta H}{\delta \Phi} \\
&=(K_2(\eta)\Phi_y)_y -(K_1(\eta)\Psi\rho)_y - \left(\frac{K_1(\eta)}{K_2(\eta)}\nu^2\sin^2(\theta_1-\theta_2)(\Phi_z-K_1(\eta)\eta_z\Phi_y)\eta_z\right)_y\nonumber \\
& \qquad\mbox{}+\left[\frac{1}{K_2(\eta)}\nu^2\sin^2(\theta_1-\theta_2)(\Phi_z-K_1(\eta)\eta_z\Phi_y)\right]_z \nonumber \\
& \qquad\mbox{}+ \nu\cos(\theta_1-\theta_2)\Psi_z- \nu\cos(\theta_1-\theta_2)(K_1(\eta)\eta_z\Psi)_y, \label{Ham eq 6}
\end{align}
where
$$K_3  = \begin{cases}
\beta, & \beta>0, \\[5mm]
\ee^y, & \beta=0,
\end{cases}$$
with boundary conditions
\begin{alignat}{2}
& - K_2(\eta)\Phi_y =0, && y = - \tfrac{1}{\beta}, \label{BC 1}\\
& - K_2(\eta)\Phi_y +K_1(\eta)\Psi\rho+\frac{K_1(\eta)}{K_2(\eta)}\nu^2\sin^2(\theta_1-\theta_2)(\Phi_z-K_1(\eta)\eta_z\Phi_y)\eta_z \nonumber \\
& \qquad\mbox{}+\nu\cos(\theta_1-\theta_2)K_1(\eta)\eta_z\Psi-\gamma(\rho\sin\theta_2+\nu\eta_z \sin\theta_1) = 0,\hspace{2cm} && y=0. \label{BC 2}
\end{alignat}
We also introduce a bifurcation parameter $\mu_1$ by
writing $\nu=\nu_0+\mu_1$, where $\nu_0$ is a reference value for $\nu$ to be chosen later.

To place equations \eqref{Ham eq 1}--\eqref{BC 2} on a rigorous footing, we introduce the spaces
\begin{align*}
H_\mathrm{per}^m(S)&=\{w\in H_\mathrm{loc}^m(\mathbb{R}):\mbox{$w(z+2\pi)=w(z)$ for all $z\in\mathbb{R}$}\},\\
H_\mathrm{per}^m(\Sigma)&=\{w\in H_\mathrm{loc}^m((-\tfrac{1}{\beta},0)\times \mathbb{R}):\mbox{$w(y,z+2\pi)=w(y,z)$ for all $(y,z)\in(-\tfrac{1}{\beta},0)\times\mathbb{R}$}\}
\end{align*}
for  $m \in {\mathbb N}_0$; recall that $H^1_\mathrm{per}(S)$ and $H^2_\mathrm{per}(\Sigma)$ are Banach algebras, while the formulae
$w \mapsto w|_{y=0}$ and $w \mapsto w|_{y=-\frac{1}{\beta}}$ (for $\beta>0$) define bounded linear mappings $H^m(\Sigma) \rightarrow H^{m-1}(S)$ for $m \in {\mathbb N}$. The following
proposition relates to mappings appearing in the above equations (see Bagri \& Groves \cite[Proposition 2.1]{BagriGroves15} for parts (i), (ii) and Buffoni \& Toland \cite{BuffoniToland} for parts (iii), (iv)).

\begin{proposition} \label{Analytic stuff}
\hspace{1in}
\begin{list}{(\roman{count})}{\usecounter{count}}
\item
The formula
$(w_1,w_2) \mapsto w_1w_2$
defines bounded bilinear mappings
$L^2_\mathrm{per}(\Sigma) \times H^1_\mathrm{per}(S) \rightarrow L^2_\mathrm{per}(\Sigma)$,
$H_\mathrm{per}^1(\Sigma) \times L^2_\mathrm{per}(S) \rightarrow L^2_\mathrm{per}(\Sigma)$
and $H_\mathrm{per}^1(\Sigma) \times H^1_\mathrm{per}(S) \rightarrow H^1_\mathrm{per}(\Sigma)$.
\item
The formula
$$(w_1,w_2) \mapsto \int_{-\frac{1}{\beta}}^0 w_1(\cdot,y)w_2(\cdot,y)\dy$$
defines bounded bilinear mappings
$L_\mathrm{per}^2(\Sigma) \times H_\mathrm{per}^1(\Sigma) \rightarrow$~$L_\mathrm{per}^2(S)$,
$H_\mathrm{per}^1(\Sigma) \times L_\mathrm{per}^2(\Sigma) \rightarrow L_\mathrm{per}^2(S)$
and $H_\mathrm{per}^1(\Sigma) \times H_\mathrm{per}^1(\Sigma) \rightarrow H_\mathrm{per}^1(S)$.
\item
The formulae
$\eta \mapsto (1+\beta \eta)^{-1}-1$ (for $\beta>0$) and $\eta \mapsto (1+\ee^y \eta)^{-1}-1$ (for $\beta=0$)
yield mappings $H^3_\mathrm{per}(S) \rightarrow H^3_\mathrm{per}(S)$ and $H^3_\mathrm{per}(S) \rightarrow H^3_\mathrm{per}(\Sigma)$
respectively which are defined and analytic in a neighbourhood of the origin.
\item
For each $n \in {\mathbb N}$ the formula $\eta \mapsto (1+(\nu_0+\mu)^2 \sin^2(\theta_1-\theta_2)\eta_z^2)^{-n}-1$ yields a mapping ${\mathbb R} \times H^3_\mathrm{per}(S) \rightarrow H^2_\mathrm{per}(S)$ which is defined and analytic in a neighbourhood of the origin.
\item
For each $n \in {\mathbb N}$ the formula $(\eta,\rho) \mapsto Q(\eta,\rho)^\frac{n}{2}$ yields a mapping $H^3_\mathrm{per}(S) \times H^2_\mathrm{per}(S) \rightarrow H^2_\mathrm{per}(S)$ which is defined and analytic in a neighbourhood of the origin.
 \end{list}
 \end{proposition} 

Let us now define
\begin{align*}
X & \!=\!\{v\!=\!(\eta,\rho,\Phi,\zeta,\xi,\Psi)\!\in\!\ H_\mathrm{per}^3(S)\times H_\mathrm{per}^2(S)\times H_\mathrm{per}^2(\Sigma)\times H_\mathrm{per}^1(S)\times H_\mathrm{per}^2(S)\times H_\mathrm{per}^1(\Sigma)\}, \\
Z & \!=\!\{v\!=\!(\eta,\rho,\Phi,\zeta,\xi,\Psi)\!\in\!\ H_\mathrm{per}^2(S)\times H_\mathrm{per}^1(S)\times H_\mathrm{per}^1(\Sigma)\times L_\mathrm{per}^2(S)\times H_\mathrm{per}^1(S)\times L_\mathrm{per}^2(\Sigma)\}.
\end{align*}
The following lemma, which is a consequence of the previous proposition, shows that the right-hand sides of equations
\eqref{Ham eq 1}--\eqref{Ham eq 6} define an analytic mapping $v_\mathrm{H}^{\mu_1}: \Lambda_1 \times U \rightarrow Z$, where
$U$ is a neighbourhood of the origin of $X$.
In the notation of the lemma, these equations define a quasilinear evolutionary system
\begin{equation}
v_x = v_\mathrm{H}^{\mu_1}(v) \label{First Ham formulation}
\end{equation}
with nonlinear boundary conditions given by
\begin{alignat}{2}
& \Phi_y = F^{\mu_1}(\eta,\rho,\Phi,\zeta,\xi,\Psi), && y = - \tfrac{1}{\beta}, \label{BC alt 1}\\
& \Phi_y + \gamma(\rho\sin\theta_2 + \nu_0\eta_z \sin\theta_1) = F^{\mu_1}(\eta,\rho,\Phi,\zeta,\xi,\Psi),\hspace{2cm} && y=0, \label{BC alt 2}
\end{alignat}
where
\begin{eqnarray*}
\lefteqn{F^{\mu_1}(\eta,\rho,\Phi,\zeta,\xi,\Psi)} \\
& &=K_2(\eta)K_3\eta\Phi_y+K_1(\eta)\Psi\rho+(\nu_0+\mu_1)\cos(\theta_1-\theta_2)K_1(\eta)\eta_z\Psi \\
&&  \qquad\mbox{}
+\frac{K_1(\eta)}{K_2(\eta)}(\nu_0+\mu_1)^2\sin^2(\theta_1-\theta_2)(\Phi_z-K_1(\eta)\eta_z\Phi_y)\eta_z - \frac{K_1(\eta)}{K_2(\eta)}\gamma\mu_1\eta_z\sin\theta_1,
\end{eqnarray*}
and we note for later use that
\begin{eqnarray*}
\lefteqn{\mathrm{d}F^{\mu_1}[\eta,\rho,\Phi,\zeta,\xi,\Psi](\tilde{\eta},\tilde{\rho},\tilde{\Phi},\tilde{\zeta},\tilde{\xi},\tilde{\Psi})} \\
& & = -K_2(\eta)^2K_3^2\eta\Phi_y\tilde{\eta}+K_2(\eta)K_3(\Phi_y\tilde{\eta}+\eta\tilde{\Phi}_y) -K_1(\eta)K_2(\eta)K_3\Psi\rho\tilde{\eta} +K_1(\eta)(\Psi\tilde{\rho}+\rho\tilde{\Psi})\\
& & \qquad\mbox{}+(\nu_0+\mu_1)\cos(\theta_1-\theta_2)\big(-K_1(\eta)K_2(\eta)K_3\eta_z\Psi\tilde{\eta} +K_1(\eta)(\Psi\tilde{\eta}_z + \eta_z\tilde{\Psi})\big)\\
& & \qquad\mbox{}+\frac{K_1(\eta)}{K_2(\eta)}(\nu_0+\mu_1)^2\sin^2(\theta_1-\theta_2) \\
& & \hspace{1.75cm}\times\!\left(\!(\Phi_z\!-\!K_1(\eta)\eta_z\Phi_y)\tilde{\eta}_z\!+\!
\eta_z\big(\tilde{\Phi}_z \!+\! K_1(\eta)K_2(\eta)K_3\eta_z\Phi_y\tilde{\eta}\!-\!K_1(\eta)(\Phi_y\tilde{\eta}_z\!+\!\eta_z\tilde{\Phi}_y)
\big)\!\right) \\
& & \qquad\mbox{}-\frac{K_1(\eta)}{K_2(\eta)}\gamma\mu_1\sin\theta_1\tilde{\eta}_z.
\end{eqnarray*}

This system is reversible; the reverser is given by
\begin{eqnarray*}
\lefteqn{R\big(\eta(z),\rho(z),\Phi(y,z),\zeta(z),\xi(z),\Psi(y,z)\big)} \qquad\\
& & =\big(\eta(-z),-\rho(-z),-\Phi(y,-z),-\zeta(-z),\xi(-z),\Psi(y,-z)\big).
\end{eqnarray*}

\begin{lemma} \label{Properties of vH and F}
There exist neighbourhoods $U$ and $\Lambda_1$ of the origin in respectively $X$ and ${\mathbb R}$ with the following properties.
\begin{list}{(\roman{count})}{\usecounter{count}}
\item
The formula $(\mu_1,v) \mapsto v_\mathrm{H}^{\mu_1}(v)$, where $v_\mathrm{H}^{\mu_1}(v)$ is defined by the right-hand sides of \eqref{Ham eq 1}--\eqref{Ham eq 6}
(with $\nu=\nu_0+\mu_1$),
defines an analytic mapping $\Lambda_1 \times U \rightarrow Z$.
\item
The formula $(\mu_1,v) \mapsto F^{\mu_1}(v)$ defines an analytic mapping $\Lambda_1 \times U \rightarrow H^1(\Sigma)$.
\item
The derivative $\mathrm{d}F^{\mu_1}[v] \in {\mathcal L}(X,H^1(\Sigma))$ has a unique extension $\widetilde{\mathrm{d}F}\vphantom{}^{\mu_1}[v] \in {\mathcal L}(Z,L^2(\Sigma))$ which
depends analytically upon $(\mu_1,v) \in \Lambda_1 \times U$.
\end{list}
\end{lemma}

It remains to confirm that \eqref{First Ham formulation} has a Hamiltonian structure; for this purpose we use the following lemma, which is proved by direct calculations and Proposition \ref{Analytic stuff}.

\begin{lemma}
\hspace{1in}
\begin{list}{(\roman{count})}{\usecounter{count}}
\item
The formula $(\mu_1,v) \mapsto H^{\mu_1}(v)$, where $H^{\mu_1}(v)$ is defined by the right-hand side of \eqref{Formula for H}
(with $\nu=\nu_0+\mu_1$),
defines an analytic mapping $\Lambda_1 \times U \rightarrow {\mathbb R}$.
\item
The derivative $\mathrm{d}H^{\mu_1}[v] \in X^\ast$ has a unique extension $\widetilde{\mathrm{d}H}\vphantom{}^{\mu_1}[v] \in Z^\ast$ which
depends analytically upon $(\mu_1,v) \in \Lambda_1 \times U$.
\item
The formula
$$\widetilde{\mathrm{d}H}\hspace{-4mm}{\hphantom{H}}^{\mu_1}[v](w) = \langle J v_\mathrm{H}^{\mu_1}(v),w\rangle,$$
where $\langle \cdot\,,\cdot \rangle$ denotes the $(L^2_\mathrm{per}(S))^6$ inner product and
$$J(\eta,\rho,\Phi,\zeta,\xi,\Psi)=(-\zeta,-\xi,-\Psi,\eta,\rho,\Phi),$$
holds for all $(\mu_1,v) \in \DD_\mathrm{H}$ and $w \in Z$, where
$$\DD_\mathrm{H} = \{(\mu_1,v) \in \Lambda_1 \times U: \mbox{\eqref{BC alt 1}, \eqref{BC alt 2} are satisfied}\},$$
so that the gradient $\nabla H^{\mu_1}(v)$ exists (and equals $J v_\mathrm{H}^{\mu_1}(v)$)
for all $(\mu_1,v) \in \DD_\mathrm{H}$ and extends to an analytic function of $(\mu_1,v) \in \Lambda_1 \times U$.
\end{list}
\end{lemma}

Altogether we conclude that $v_\mathrm{H}^{\mu_1}(v)=J^{-1}\nabla H^{\mu_1}(v)$ for $(\mu_1,v) \in \DD_\mathrm{H}$ defines the Hamiltonian vector field for the 
Hamiltonian system $(Z,\Omega,H^{\mu_1})$, where $\Omega: Z^2 \rightarrow {\mathbb R}$ is the constant symplectic $2$-form
\begin{eqnarray*}
\lefteqn{\Omega\big((\eta_1,\rho_1,\Phi_1,\zeta_1,\xi_1,\Psi_1),(\eta_2,\rho_2,\Phi_2,\zeta_2,\xi_2,\Psi_2)\big)} \\
& & = \langle J(\eta_1,\rho_1,\Phi_1,\zeta_1,\xi_1,\Psi_1),(\eta_2,\rho_2,\Phi_2,\zeta_2,\xi_2,\Psi_2)\rangle \\
& & =\int_S(\zeta_2\eta_1-\eta_2\zeta_1+\xi_2\rho_1-\rho_2\xi_1)\dz+\int_\Sigma(\Psi_2\Phi_1-\Phi_2\Psi_1)\dy\dz.
\end{eqnarray*}
Note that $\Omega$ is the exterior derivative of the parameter-independent $1$-form $\omega|_v$ given by
\begin{align*}
\omega|_{(\eta,\rho,\Phi,\zeta,\xi,\Psi)}(\tilde\eta,\tilde\rho,\tilde\Phi,\tilde\zeta,\tilde\xi,\tilde\Psi)&=\int_S (\eta\tilde\zeta+\rho\tilde\xi)\dz + \int_\Sigma \Phi\tilde\Psi \dy\dz \\
& = \langle \alpha(\eta,\rho,\Phi,\zeta,\xi,\Psi), (\tilde\eta,\tilde\rho,\tilde\Phi,\tilde\zeta,\tilde\xi,\tilde\Psi)\rangle,
\end{align*}
where
$$\alpha(\eta,\rho,\Phi,\zeta,\xi,\Psi)=(0,0,0,\eta,\rho,\Phi).$$

The system \eqref{First Ham formulation}--\eqref{BC alt 2} is unsuitable for analysis due to its nonlinear boundary conditions. We proceed by
replacing $\Phi$ with the new variable
$$\Gamma=\Phi-\partial_y \Delta^{-1} F^{\mu_1}(\eta,\rho,\Phi,\zeta,\xi,\Psi),$$
where $\Delta$ is the Dirichlet Laplacian in $\Sigma$. In the notation of the following lemma we find that
for $(\eta,\rho,\zeta,\xi,\Psi) \in W$ and $\mu_1 \in \Lambda_1$ the variable $\Phi \in V_1$ satisfies
\eqref{BC alt 1}, \eqref{BC alt 2} if and only if the variable $\Gamma \in V_2$ satisfies
\pagebreak
\begin{alignat}{2}
& \Gamma_y = 0, && y = - \tfrac{1}{\beta}, \label{Linear BC 1} \\
& \Gamma_y + \gamma(\rho\sin\theta_2 + \nu_0\eta_z \sin\theta_1) = 0,\hspace{2cm} && y=0, \label{Linear BC 2}
\end{alignat}
because
\begin{align*}
\Gamma_y &= \Phi_y - \partial_{yy} \Delta^{-1}F^{\mu_1}(\eta,\rho,\Phi,\zeta,\xi,\Psi) \\
&= \Phi_y - \partial_{yy} \Delta^{-1}F^{\mu_1}(\eta,\rho,\Phi,\zeta,\xi,\Psi) - \underbrace{\partial_{zz} \Delta^{-1}F^{\mu_1}(\eta,\rho,\Phi,\zeta,\xi,\Psi)}_{\displaystyle=0} \\
& = \Phi_y - F^{\mu_1}(\eta,\rho,\Phi,\zeta,\xi,\Psi) 
\end{align*}
for $y=0$ and $y=-\frac{1}{\beta}$.

\begin{lemma}\label{change-of-variables}
\hspace{1in}
\begin{list}{(\roman{count})}{\usecounter{count}}
\item
There exist neighbourhoods $V_1$, $V_2$ of the origin in $H_\mathrm{per}^2({\mathbb R})$ and $W$ of the origin in
$$
X_0 = \{(\eta,\rho,\zeta,\xi,\Psi)\!\in\!\ H_\mathrm{per}^3(S)\times H_\mathrm{per}^2(S)\times H_\mathrm{per}^1(S)\times H_\mathrm{per}^2(S)\times H_\mathrm{per}^1(\Sigma)\}
$$
such that $\Phi \mapsto \Gamma(\Phi,
(\eta,\rho,\zeta,\xi,\Psi),\mu_1)$ is an analytic diffeomorphism $V_1 \rightarrow V_2$ which, together with its inverse,
depends analytically upon $(\eta,\rho,\zeta,\xi,\Psi) \in W$ and $\mu_1 \in \Lambda_1$.
\item
The derivative $\mathrm{d}_1\Gamma[\Phi,(\eta,\rho,\zeta,\xi,\Psi),\mu_1] \in {\mathcal L}(H_\mathrm{per}^2(\Sigma))$ extends to an isomorphism in ${\mathcal L}(H_\mathrm{per}^1(\Sigma))$ which, together with its inverse,
depends analytically upon $\Phi \in V_1$,\linebreak
$(\eta,\rho,\zeta,\xi,\Psi) \in W$ and $\mu_1 \in \Lambda_1$.
\end{list}
\end{lemma} 
\begin{proof}
(i) This result follows by applying the implicit-function theorem to the equation
$$
g(\Phi, \Gamma, (\eta,\rho,\zeta,\xi,\Psi),\mu_1):=
\Gamma-\Gamma(\Phi,(\eta,\rho,\zeta,\xi,\Psi),\mu_1)=0.
$$
Here we note that $g$ maps (a neighbourhood of the origin in) $H_\mathrm{per}^2(\Sigma) \times H_\mathrm{per}^2(\Sigma) \times X_0 \times {\mathbb R}$ into
$H_\mathrm{per}^2(\Sigma)$ (by Lemma \ref{Properties of vH and F}(ii) and the fact that $\Delta^{-1}$ belongs to ${\mathcal L}(H^1_\mathrm{per}(\Sigma),H^3_\mathrm{per}(\Sigma))$),
and that $g(0,0,0,0)=0$ and $\mathrm{d}_1g[0,0,0,0]=-I$.

(ii) It follows from Lemma \ref{Properties of vH and F}(iii) and the fact that $\Delta^{-1}$ belongs to ${\mathcal L}(L^2_\mathrm{per}(\Sigma),H^2_\mathrm{per}(\Sigma))$
that
$\mathrm{d}_1\Gamma[\Phi,(\eta,\rho,\zeta,\xi,\Psi),\mu_1] \in {\mathcal L}(H_\mathrm{per}^2(\Sigma))$ extends to an element $\widetilde{\mathrm{d}_1\Gamma}[\Phi,(\eta,\rho,\zeta,\xi,\Psi),\mu_1]$ of ${\mathcal L}(H_\mathrm{per}^1(\Sigma))$ which
depends analytically upon $\Phi \in V_1$, $(\eta,\rho,\zeta,\xi,\Psi) \in W$ and $\mu_1 \in \Lambda_1$. Obviously $\widetilde{\mathrm{d}_1\Gamma}[0,0,0,0]=I$ is an isomorphism, which is an open property. The analyticity of $\widetilde{\mathrm{d}_1\Gamma}[\cdot]$ implies the analyticity of its inverse.
\end{proof}

It follows from the above lemma that the formula
$$G^{\mu_1}(\eta,\rho,\Phi,\zeta,\xi,\Psi)=(\eta,\rho,\Gamma,\zeta,\xi,\Psi)$$
defines a valid change of variable: it is an analytic diffeomorphism from $U$ to a neighbourhood $\hat{U}$ of the origin in $X$, the
operator $\mathrm{d}G^{\mu_1}[u] \in {\mathcal L}(X)$ extends to an isomorphism $\widetilde{\mathrm{d}G}\hphantom{}^{\mu_1}[u] \in {\mathcal L}(Z)$,
and $G^{\mu_1}$, $\widetilde{\mathrm{d}G}\hphantom{}^{\mu_1}[u]$ and their inverses depend analytically upon $(u,\mu_1) \in U \times \Lambda_1$.
The system \eqref{First Ham formulation}--\eqref{BC alt 2} is transformed into
\begin{equation}\label{Final sd formulation}
\hat{v}_x=\hat{v}_\mathrm{H}^{\mu_1}(\hat{v}),
\end{equation}
where
\begin{equation*}
\hat{v}_\mathrm{H}^{\mu_1}(\hat{v})=\widetilde{\mathrm{d}G}\hphantom{}^{\mu_1}[(G^{\mu_1})^{-1}(\hat{v})](v_\mathrm{H}^{\mu_1}((G^{\mu_1})^{-1}(\hat{v})))
\end{equation*}
with linear boundary conditions \eqref{Linear BC 1}, \eqref{Linear BC 2}. Note also that $G^{\mu_1}$ and $(G^{\mu_1})^{-1}$ both commute with the reverser $R$, so that
\eqref{Final sd formulation} inherits the reversibility of equation \eqref{First Ham formulation}.

Writing $(G^{\mu_1})^{-1}$ as $K^{\mu_1}$, one finds that
the change of variable transforms $(Z,\Omega,H^{\mu_1})$ into the new Hamiltonian system $(Z,\hat{\Omega}^\mu,\hat{H}^\mu)$, where
\begin{equation}
\hat{\Omega}^{\mu_1}|_{\hat{v}}(\hat{v}_1,\hat{v}_2)= \langle \hat{J}^{\mu_1}(\hat{v}) \hat{v}_1,\hat{v}_2 \rangle
\label{Definition of hydro Omega}
\end{equation}
with
\begin{equation*}
\hat{J}^{\mu_1}(\hat{v})=\widetilde{\mathrm{d}K}\hphantom{}^{\mu_1}[\hat{v}]^\ast J\widetilde{\mathrm{d}K}\hphantom{}^{\mu_1}[\hat{v}]
\end{equation*}
and
\begin{equation}
\hat{H}^{\mu_1}(\hat{v})=H^{\mu_1}(K^{\mu_1}(\hat{v}))
\label{Definition of hydro H}
\end{equation}
for $(\mu_1,\hat{v}) \in \Lambda_1 \times \hat{U}$.
In particular 
$$\hat{v}_\mathrm{H}^{\mu_1}(\hat{v})=\hat{J}^{\mu_1}(\hat{v})\nabla \hat{H}^{\mu_1}(\hat{v})$$
for $(\mu_1,\hat{v}) \in \hat{\DD}_\mathrm{H}$, where
$$\hat{\DD}_\mathrm{H} = \{(\mu_1,\hat{v}) \in \Lambda_1 \times \hat{U}: \mbox{\eqref{Linear BC 1}, \eqref{Linear BC 2} are satisfied}\}.$$
Note further that $\hat{\Omega}^{\mu_1}|_{\hat{v}}$ is the exterior derivative of the $1$-form $\hat{\omega}^{\mu_1}|_{\hat{v}}$ given by
$$\hat{\omega}^{\mu_1}|_{\hat{v}}(\hat{w}) = \langle \hat{\alpha}^{\mu_1}(\hat{v}),\hat{w} \rangle$$
with
$$
\hat{\alpha}^{\mu_1}(\hat{v})=\widetilde{\mathrm{d}K}\hphantom{}^{\mu_1}[\hat{v}]^\ast(\alpha(K^{\mu_1}(\hat{v}))).
$$
The validity of these calculations relies upon the existence of the adjoint operator $\widetilde{\mathrm{d}K}\hphantom{}^{\mu_1}[\hat{v}]^\ast$, and this
assumption is verified in the following result.

\begin{proposition}\label{adjoint_proposition}
The adjoint operators
$\widetilde{\mathrm{d}G}\vphantom{}^{\mu_1}[v]^\ast$, $\widetilde{\mathrm{d}K}\vphantom{}^{\mu_1}[\hat{v}]^\ast \in {\mathcal L}(Z)$ exist and
depend analytically upon $({\mu_1},v) \in \Lambda_1 \times U$ and $({\mu_1},\hat{v}) \in \Lambda_1 \times \hat{U}$.
\end{proposition}
\begin{proof} The existence of the adjoint $\widetilde{\mathrm{d}G}\vphantom{}^{\mu_1}[v]^\ast \in {\mathcal L}(Z)$ follows by a direct
calculation; its components are given by
\pagebreak
\begin{align*}
(\widetilde{\mathrm{d}G}\hphantom{}^{\mu_1}[v]^\ast(\tilde{v}))_\eta&=\tilde{\eta}+\!\!\int_{-\frac{1}{\beta}}^0\!\!\bigg\{\!\bigg(\!\!-K_2(\eta)^2K_3^2\eta\Phi_y+K_2(\eta)K_3\Phi_y
\\
&\hspace{0.95in}\mbox{} -K_1(\eta)K_2(\eta)K_3\Psi\rho-(\nu_0+\mu_1)\cos(\theta_1-\theta_2) K_1(\eta)K_2(\eta)K_3\eta_z\Psi\\
&\hspace{0.95in}\mbox{}+\frac{K_1(\eta)}{K_2(\eta)}(\nu_0+\mu_1)^2\sin^2(\theta_1-\theta_2)\eta_zK_1(\eta)K_2(\eta)K_3\eta_z\Phi_y\bigg)\Delta^{-1}(\tilde{\Gamma}_y)\\
&\hspace{0.8in}\mbox{} 
-\bigg(\!\bigg((\nu_0+\mu_1)\cos(\theta_1-\theta_2)K_1(\eta)\Psi\\
&\hspace{1.2in}\mbox{} +\frac{K_1(\eta)}{K_2(\eta)}(\nu_0+\mu_1)^2\sin^2(\theta_1-\theta_2)(\Phi_z-K_1(\eta)\eta_z\Phi_y)\\
&\hspace{1.2in}\mbox{} -\frac{K_1^2(\eta)}{K_2(\eta)}(\nu_0+\mu_1)^2\sin^2(\theta_1-\theta_2)\eta_z\Phi_y\\
&\hspace{1.2in}\mbox{} -\frac{K_1(\eta)}{K_2(\eta)}\gamma\mu_1\sin \theta_1\bigg)\Delta^{-1}(\tilde{\Gamma}_y)\bigg)_z\bigg\}\ \dy,\\
(\widetilde{\mathrm{d}G}\hphantom{}^{\mu_1}[v]^\ast(\tilde{v}))_\rho&=\tilde\rho+\int_{-\frac{1}{\beta}}^0K_1(\eta)\Psi\Delta^{-1}(\tilde{\Gamma}_y)\dy,\\
(\widetilde{\mathrm{d}G}\hphantom{}^{\mu_1}[v]^\ast(\tilde{v}))_\Gamma&=\tilde\Gamma-\bigg(\!\bigg(K_2(\eta)K_3\eta-\frac{K_1^2(\eta)}{K_2(\eta)}(\nu_0+\mu_1)^2\sin^2(\theta_1-\theta_2)\eta_z^2\bigg)\Delta^{-1}(\tilde{\Gamma}_y)\bigg)_y,\\
&\qquad\mbox{}-\bigg(\frac{K_1(\eta)}{K_2(\eta)}(\nu_0+\mu_1)^2\sin^2(\theta_1-\theta_2)\eta_z\Delta^{-1}(\tilde{\Gamma}_y)\bigg)_z,\\
(\widetilde{\mathrm{d}G}\hphantom{}^{\mu_1}[v]^\ast(\tilde{v}))_\zeta&=\tilde{\zeta},\\
(\widetilde{\mathrm{d}G}\hphantom{}^{\mu_1}[v]^\ast(\tilde{v}))_\xi&=\tilde{\xi},\\
(\widetilde{\mathrm{d}G}\hphantom{}^{\mu_1}[v]^\ast(\tilde{v}))_\Psi&=\tilde\Psi+\bigg(K_1(\eta)\rho+(\nu_0+\mu_1)\cos(\theta_1-\theta_2)K_1(\eta)\eta_z\bigg)\Delta^{-1}(\tilde{\Gamma}_y),
\end{align*}
from which the analyticity of $\widetilde{\mathrm{d}G}\hphantom{}^{\mu_1}[v]$ also follows by Proposition \ref{Analytic stuff}. (Note that the formula 
$(\Delta^{-1})^\ast=\Delta^{-1}$ has been used in this calculation.)

Observe that $\widetilde{\mathrm{d}G}\hphantom{}^{\mu_1}[v]^\ast$ is an isomorphism because $\widetilde{\mathrm{d}G}\hphantom{}^0[0]^\ast=I$ is obviously an isomorphism, which is an open property. Furthermore the analytic dependence of $\widetilde{\mathrm{d}G}\hphantom{}^{\mu_1}[v]^\ast$ upon $(\mu_1,v) \in \Lambda_1 \times U$ implies the same of its inverse, and the calculation
\begin{align*}
\langle \widetilde{\mathrm{d}G}\hphantom{}^{\mu_1}[v]^{-1} (v_1), v_2 \rangle &= \langle \widetilde{\mathrm{d}G}\hphantom{}^{\mu_1}[v]^{-1} v_1, \widetilde{\mathrm{d}G}\hphantom{}^{\mu_1}[v]^\ast (\widetilde{\mathrm{d}G}\hphantom{}^{\mu_1}[v]^\ast)^{-1} (v_2) \rangle \\
& = \langle \widetilde{\mathrm{d}G}\hphantom{}^{\mu_1}[v]\widetilde{\mathrm{d}G}\hphantom{}^{\mu_1}[v]^{-1} (v_1), (\widetilde{\mathrm{d}G}\hphantom{}^{\mu_1}[v]^\ast)^{-1} (v_2) \rangle \\
& = \langle v_1, (\widetilde{\mathrm{d}G}\hphantom{}^{\mu_1}[v]^\ast)^{-1} (v_2) \rangle
\end{align*}
shows that $(\widetilde{\mathrm{d}G}\hphantom{}^{\mu_1}[v]^{-1})^\ast$ exists and equals $(\widetilde{\mathrm{d}G}\hphantom{}^{\mu_1}[v]^\ast)^{-1}$. The proof is completed by
noting that $\widetilde{\mathrm{d}K}\hphantom{}^{\mu_1}[\hat{v}]=\widetilde{\mathrm{d}G}\hphantom{}^{\mu_1}[v]^{-1}$ with $v=(G^{\mu_1})^{-1}(\hat{v})$.
\end{proof}

\section{Application of Lyapunov centre theory to hydroelastic waves} \label{Application}

In this section we apply Theorem \ref{LCT} to the spatial dynamics formulation
$$
\hat{v}_x=\hat{v}_\mathrm{H}^{\mu_1}(\hat{v})
$$
for hydroelastic waves derived in Section \ref{Formulation}. For this purpose we define
\begin{align*}
X & \!=\!\{\hat{v}\!=\!(\eta,\rho,\Gamma,\zeta,\xi,\Psi)\!\in\!\ H_\mathrm{per}^3(S)\times H_\mathrm{per}^2(S)\times H_\mathrm{per}^2(\Sigma)\times H_\mathrm{per}^1(S)\times H_\mathrm{per}^2(S)\times H_\mathrm{per}^1(\Sigma):\\
& \hspace{2in}\Gamma_y\big|_{y=-\frac{1}{\beta}} = 0,\ \Gamma_y\big|_{y=0} + \gamma(\rho\sin\theta_2 + \nu_0\eta_z \sin\theta_1) = 0\}, \\
Z & \!=\!\{v\!=\!(\eta,\rho,\Gamma,\zeta,\xi,\Psi)\!\in\!\ H_\mathrm{per}^2(S)\times H_\mathrm{per}^1(S)\times H_\mathrm{per}^1(\Sigma)\times L_\mathrm{per}^2(S)\times H_\mathrm{per}^1(S)\times L_\mathrm{per}^2(\Sigma)\}
\end{align*}
(note the modification to the space $X$)
and consider $\hat{v}_\mathrm{H}^{\mu_1}$ as the Hamiltonian vector field for the Hamiltonian system $(Z, \hat{\Omega}^{\mu_1}, \hat{H}^{\mu_1})$, where $\hat{\Omega}^{\mu_1}$ and $\hat{H}^{\mu_1}$ are defined in equations \eqref{Definition of hydro Omega}
and \eqref{Definition of hydro H} and $\DD_H=\Lambda_1 \times \hat{U}$ is a neighbourhood of the origin in ${\mathbb R} \times X$.
Defining
$L^{\mu_1}=\mathrm{d}\hat{v}_\mathrm{H}^{\mu_1}[0]$ and $N^{\mu_1}(\hat{v})=\hat{v}_\mathrm{H}^{\mu_1}(\hat{v})-L^{\mu_1}\hat{v}$,
one can write Hamilton's equations as
\begin{equation}
\hat{v}_x = L^{\mu_1}\hat{v}+N^{\mu_1}(\hat{v}), \label{Apply theorem to this}
\end{equation}
where
in particular $L:=L^0$ is given by the explicit formula
$$
L\begin{pmatrix}
\eta \\ \rho \\ \Gamma \\ \zeta \\ \xi \\ \Psi
\end{pmatrix}
=
\begin{pmatrix}
\rho\\[1mm]
\xi - \nu_0^2\eta_{zz} - 2\nu_0 \cos (\theta_1-\theta_2)\rho_z \\[1mm]
\Psi - \nu_0\cos(\theta_1-\theta_2) \Gamma_z \\[1mm]
\nu_0^2\xi_{zz} +  \eta - \gamma\nu_0\sin \theta_1 \Gamma_z|_{y=0} \\[1mm]
-\zeta+\gamma \sin\theta_2 \Gamma|_{y=0} - 2 \nu_0\cos(\theta_1-\theta_2)\xi_z \\[1mm]
-\Gamma_{yy} - \nu_0^2\sin^2(\theta_1-\theta_2)\Gamma_{zz}-\nu_0\cos(\theta_1-\theta_2)\Psi_z
\end{pmatrix},
$$
which is readily calculated from \eqref{Ham eq 1}--\eqref{Ham eq 6} since the change of variable used to linearise the boundary condition is near-identity.
Equation \eqref{Apply theorem to this} evidently satisfies hypothesis (H1).

Furthermore, the reverser $R$ clearly satisfies $R^\ast=R$ and 
$$H^{\mu_1}(Rv) = H^{\mu_1}(v), \quad R^\ast\alpha^{\mu_1}(Rv)= -\alpha^{\mu_1}(v), \quad R^\ast J^{\mu_1}(Rv)R =-J^{\mu_1}(v)$$
for all $(\mu_1,v) \in \Lambda_1 \times U$; since $R$ commutes with $G^{\mu_1}$ and $K^{\mu_1}$ we conclude that
$$\hat{H}^{\mu_1}(R\hat{v}) = \hat{H}^{\mu_1}(\hat{v}), \quad R^\ast\hat{\alpha}^{\mu_1}(R\hat{v})= -\hat{\alpha}^{\mu_1}(\hat{v}), \quad R^\ast \hat{J}^{\mu_1}(R\hat{v})R =-\hat{J}^{\mu_1}(\hat{v})$$
for all $(\mu_1,\hat{v}) \in \Lambda_1 \times \hat{U}$. Hypothesis (H2) is therefore also satisfied.

\subsection{Purely imaginary spectrum} \label{pis}
The next step is to examine the purely imaginary spectrum of the linear operator $L$. This task is readily accomplished by
using Fourier-series representations
$$\hat{v}(y,z) = \sum_{k \in {\mathbb Z}} \hat{v}_k(y) \ee^{\ii kz}, \qquad \hat{v}^\star(y,z) = \sum_{k \in {\mathbb Z}} \hat{v}_k^\star(y) \ee^{\ii kz}$$
for $\hat{v} \in X$ and $\hat{v}^\star \in Z$ and examining the resulting decoupled spectral problems for each Fourier mode.
We begin with the following lemma, which is proved by well-established methods for spatial dynamics problems
(see Groves \& Haragus
\cite{GrovesHaragus03}, Bagri \& Groves \cite{BagriGroves15} and the references therein).

\begin{lemma}\label{lemma_spectrum}
\hspace{1in}
\begin{list}{(\roman{count})}{\usecounter{count}}
\item
Suppose that $s \in {\mathbb R}$ and $k \in {\mathbb Z}$ are not both zero. The imaginary number $\ii s$ is a mode $k$ eigenvalue of $L$ if and only if
\begin{equation}
(1+\sigma_k^4)\sigma_k - \frac{\gamma^2(k\nu_0 \sin \theta_1+s \sin \theta_2)^2}{\tanh(\beta^{-1}\sigma_k)} =0,
\label{Imaginary eigenvalue equation}
\end{equation}
where
$$\sigma_k^2 = s^2 + 2 k \nu_0 s \cos(\theta_1-\theta_2)+k^2 \nu_0^2 >0,$$
and in this case the corresponding eigenvector is $\hat{v}_{k,s} \ee^{\ii k z}$, where
\small
\begin{equation}
\hat{v}_{k,s}=\begin{pmatrix}
\dfrac{ \ii\gamma b_k}{1 +\sigma_k^4} \\[5mm]
-\dfrac{ s\gamma b_k}{1 +\sigma_k^4} \\[5mm]
\frac{1}{2} \ee^{-\sigma_k y} (1-t_k)
 +\frac{1}{2} \ee^{\sigma_k y}(1+t_k)\\[5mm]
 \gamma\sin \theta_2-\dfrac{\gamma \sigma_k^2 a_{2k} b_k}{1 +\sigma_k^4}\\[5mm]
-\dfrac{\ii \gamma\sigma_k^2 b_k}{1 +\sigma_k^4} \\[5mm]
\frac{1}{2} \ii a_k
\Big(\ee^{-\sigma_k y} (1-t_k)
+\ee^{\sigma_k y}(1+t_k)\Big)
\end{pmatrix} \label{evec}
\end{equation}
\normalsize
and
$$t_k=\tanh(\beta^{-1}\sigma_k), \quad a_k = s+k \nu_0 \cos(\theta_1-\theta_2), \quad
b_k =k \nu_0  \sin \theta_1+s \sin \theta_2>0.$$
This eigenvalue has a Jordan chain of length at least $2$ with generalised eigenvector $\hat{w}_{k,s} \ee^{\ii k z}$, where
\small
$$\hat{w}_{k,s}=\begin{pmatrix}
\dfrac{\gamma\sin\theta_2}{1+\sigma_k^4}-2\gamma a_kb_k\left(\dfrac{2\sigma_k^2}{(1+\sigma_k^4)^2}+\dfrac{c_k}{2\sigma_k^2(1+\sigma_k^4)}\right) \\[5mm]
\dfrac{\ii\gamma b_k}{1+\sigma_k^4}+\dfrac{\ii s \gamma\sin \theta_2}{1+\sigma_k^4}-2\ii \gamma sa_kb_k\left(\dfrac{2\sigma_k^2}{(1+\sigma_k^4)^2}+\dfrac{c_k}{2\sigma_k^2(1+\sigma_k^4)}\right) \\[5mm]
\dfrac{\ii a_k}{2\sigma_k}\big((1-t_k)y \ee^{-\sigma_k y}-(1+t_k)y \ee^{\sigma_k y}\big)+\dfrac{\ii c_ka_k}{2\sigma_k^2}(\ee^{-\sigma_k y}+\ee^{\sigma_k y})\\[5mm]
\ii\left(\dfrac{c_ka_k}{\sigma_k^2}+ \dfrac{\sigma_k^2a_{2k}}{1+\sigma_k^4}\right)\gamma\sin\theta_2 + \dfrac{\ii\gamma  \sigma_k^2b_k}{1+\sigma_k^4}
-2\ii\gamma   a_k a_{2k} b_k \left(\dfrac{2\sigma_k^4}{(1+\sigma_k^4)^2}+\dfrac{c_k-2}{2(1+\sigma_k^4)}\right) \\[5mm]
-\dfrac{\sigma_k^2}{1+\sigma_k^4}\gamma\sin \theta_2 +2 \gamma a_kb_k \left(\dfrac{2\sigma_k^4}{(1+\sigma_k^4)^2}
+\dfrac{c_k-2}{2(1+\sigma_k^4)}\right)\\[5mm]
\left(\frac{1}{2}(1-t_k)-\dfrac{c_ka_k^2}{2\sigma_k^2}-\dfrac{(1-t_k)a_k^2}{2\sigma_k}y\right)\ee^{-\sigma_k y}+\left(\frac{1}{2}(1+t_k)-\dfrac{c_ka_k^2}{2\sigma_k^2}+\dfrac{(1+t_k)a_k^2y}{2\sigma_k}\right)\ee^{\sigma_k y}
\end{pmatrix}
$$
\normalsize
and
$$c_k=2\beta^{-1}\sigma_k \cosech (2\beta^{-1}\sigma_k),$$
if either
\begin{itemize}
\item[(a)] $\beta>0$, $s=a\beta$, $\nu_0 = \tilde{\nu}_0 \beta$, $(5+\tilde c_k)\tilde a_k\tilde b_k-2\sin\theta_2\tilde\sigma_k^2\neq 0$ and $(\beta,\gamma)$ lies on a point of the curve
$$C_k = \{(\beta_k(a),\gamma_k(a)): a \in (0,\infty)\},$$
where
\begin{align*}
\beta^4_k(a) &= \frac{1}{\tilde\sigma_k^4}\cdot \frac{2\sin\theta_2\tilde\sigma_k^2-(1+\tilde c_k)\tilde a_k\tilde b_k}{(5+\tilde c_k)\tilde a_k\tilde b_k-2\sin\theta_2\tilde\sigma_k^2}, \\
\gamma^2_k(a) & = \frac{(1+\beta_k^4(a)\tilde\sigma_k^4)\tilde\sigma_k\tanh(\tilde\sigma_k)}{\beta_k(a)\tilde b_k^2}
\end{align*}
and
\begin{align*}
	\tilde{\sigma}_k^2 &= a^2 + 2 k \tilde{\nu}_0 a \cos(\theta_1-\theta_2)+k^2 \tilde{\nu}_0^2, \\
	\tilde a_k &= a+k\tilde\nu_0\cos(\theta_1-\theta_2), \\
	\tilde b_k &= a\sin\theta_2+k\tilde\nu_0\sin\theta_1, \\
	\tilde c_k &= 2\tilde\sigma_k \cosech(2\tilde \sigma_k);
\end{align*}
\item[(b)]
 $\beta>0$, $s=a\beta$, $\nu_0 = \tilde{\nu}_0 \beta$ and
$(5+\tilde c_k)\tilde a_k\tilde b_k-2\sin\theta_2\tilde\sigma_k^2=0$, which implies that $\theta_2=0$ and $a=-k\tilde{\nu_0}\cos(\theta_1)$;
\item[(c)]
$\beta=0$ and
$$2\sin\theta_2-\frac{4\sigma_k^2a_kb_k}{1+\sigma_k^4}=\frac{a_kb_k}{\sigma_k^2}.$$
\end{itemize}
\item
Suppose $\beta>0$. Zero is a mode $0$ eigenvalue of $L$ with a Jordan chain of length 2 if $\gamma^{-2} \neq \beta \sin^2 \theta_2$ and length $4$ if $\gamma^{-2} = \beta \sin^2 \theta_2$; the generalised eigenvectors are
$$\hat{f}_1=\begin{pmatrix} 0 \\ 0 \\ 1 \\\gamma\sin \theta_2 \\ 0 \\ 0 \end{pmatrix},
\ \hat{f}_2=\begin{pmatrix} \gamma\sin \theta_2 \\ 0 \\ 0 \\ 0 \\ 0 \\ 1 \end{pmatrix},
\ \hat{f}_3=\begin{pmatrix} 0 \\ \gamma\sin \theta_2 \\ -\frac{1}{2}y^2-\beta^{-1}y \\ 0 \\ 0 \\ 0 \end{pmatrix},
\ \hat{f}_4=\begin{pmatrix} 0 \\ 0 \\ 0 \\ 0 \\ \gamma\sin \theta_2 \\ -\frac{1}{2}y^2-\beta^{-1}y \end{pmatrix},
$$
where $L\hat{f}_1=0$, $L\hat{f}_2=\hat{f}_1$ and $L\hat{f}_3=\hat{f}_2$, $L\hat{f}_4=\hat{f}_3$ if $\gamma^{-2} = \beta \sin^2 \theta_2$.
\item
Suppose $\beta=0$ and that $s=0$ does not solve \eqref{Imaginary eigenvalue equation} for any $k \in {\mathbb Z} \setminus \{0\}$ (so that zero is not a
mode $k$ eigenvalue for any $k \in {\mathbb Z} \setminus \{0\}$). Zero is not a mode $0$ eigenvalue of $L$, which instead has essential spectrum at the origin. More precisely,
the equation $L\hat{v}=\hat{v}^\star$ has a unique solution for each $\hat{v}^\star \in Z$ which satisfies the regularity requirement
\begin{equation*}
\int_{-\infty}^y \int_{-\infty}^t \Psi_0^\star(s) \ds\dt, \ \int_{-\infty}^y \Psi_0^\star(t)\dt \in L^2(-\infty,0)
\end{equation*}
and compatibility condition
\begin{equation*}
\int_{-\infty}^0 \Psi_0^\star(t)\dt - \gamma\sin \theta_2 \eta_0^\star=0.
\end{equation*}
This solution satisfies the estimate
$$\|\hat{v}-\bl \hat{v} \br_0\|_X \lesssim \|\hat{v}^\star\|_Z$$
and its $0$th Fourier component is given by the formula
$$
\begin{pmatrix}
\hat{\eta}_0 \\ \hat{\rho}_0 \\ \hat{\Gamma}_0 \\ \hat{\zeta}_0 \\ \hat{\xi}_0 \\ \hat{\Psi}_0
\end{pmatrix}
=\begin{pmatrix}
\zeta_0^\star \\ \eta_0^\star \\ -\int_{-\infty}^y \int_{-\infty}^t \Psi_0^\star(s) \ds\dt \\ -\xi_0^\star-\gamma\sin\theta_2\int_{-\infty}^0 \int_{-\infty}^t \Psi_0^\star(s) \ds\dt \\ \rho_0^\star \\ \Gamma_0^\star
\end{pmatrix}.
$$
\item
Suppose that $s \in {\mathbb R} \setminus \{0\}$ does not satisfy \eqref{Imaginary eigenvalue equation} for any $k \in {\mathbb Z}$. The imaginary number $\ii s$ belongs to the resolvent set of $L$.
\item
The resolvent estimates
$$\|(\ii s I-L)^{-1}\|_{Z \rightarrow X} \lesssim 1, \qquad \|(\ii s I-L)^{-1}\|_{Z \rightarrow Z} \lesssim \frac{1}{|s|}$$
hold uniformly over all sufficiently large values of $|s|$.
\end{list}

\end{lemma}

We proceed by interpreting equation \eqref{Imaginary eigenvalue equation} geometrically. Let
\begin{align}
\ell_1&=s\sin \theta_2+\nu_0k\sin \theta_1,\label{l1eq}\\
\ell_2&=-s\cos\theta_2 -\nu_0 k \cos\theta_1,\label{l2eq}
\end{align}
and note that
$\ell_1^2+\ell_2^2=\sigma_k^2$, so that \eqref{Imaginary eigenvalue equation} can be written as
\begin{equation}\label{disprelgeo}
\DD(\ell_1,\ell_2):=(1+(\ell_1^2+\ell_2^2)^2)\sqrt{\ell_1^2+\ell_2^2}\tanh\left( \beta^{-1}\sqrt{\ell_1^2+\ell_2^2}\right)-\gamma^2\ell_1^2=0.
\end{equation}
A mode $k$ purely imaginary eigenvalue $\ii s$ (with $(k,s) \neq (0,0)$) therefore corresponds to an intersection in the $(\ell_1,\ell_2)$-plane of
the dispersion curve
$$C_\mathrm{dr}=\{(\ell_1,\ell_2) \in {\mathbb R}^2 \setminus \{(0,0)\}: \DD(\ell_1,\ell_2)=0\}$$
with the straight line $S_k$ defined by equations \eqref{l1eq}, \eqref{l2eq}. (The solution $(\ell_1,\ell_2)=(0,0)$ of $\DD(\ell_1,\ell_2)=0$ is excluded since it corresponds to $(k,s)=(0,0)$.)

The dispersion curve $C_\mathrm{dr}$ is described parametrically by
\begin{equation*}
C_\mathrm{dr}:=\Big\{(\ell_1,\ell_2)\in\mathbb{R}^2:\ell_1^2=\frac{(1+a^4)a}{\gamma^2}\tanh(\beta^{-1} a),\ \ell_2^2=a^2-\frac{(1+a^4)a}{\gamma^2}\tanh(\beta^{-1}a),\ a>0\Big\};
\end{equation*}
its shape is shown in Figure~\ref{Intersections} (a, insets) in the 
indicated regions of the
$(\beta,\gamma)$-parameter plane. The delimiting curves are
 $$D_1=\{(\beta_0(a),\gamma_0(a))\big|_{k=0,\, \theta_2=\frac{\pi}{2}}: a \in (0,\infty)\},$$
at each point of which
the equation ${\mathcal D}(a\beta,0)=0$ has double roots $\pm a$, and
$$D_2=\{(\beta,\beta^{-1/2}): \beta \geq 0\},$$
at each point of which the equation ${\mathcal D}(\ell_1,0)=0$ has a double zero root.
We find that $C_\mathrm{dr}=\emptyset$ in the region below the curve $D_1$.
In the region between the curves $D_1$ and $D_2$ the equation ${\mathcal D}(\ell_1,0)=0$ has two pairs of simple nonzero
roots $\pm \ell_1^{(1)}$, $\pm \ell_1^{(2)}$; the branches of $C_\mathrm{dr}$ intersect the $\ell_1$ axis
vertically at the points $(\pm \ell_1^{(1)},0)$ and $(\pm \ell_1^{(2)},0)$. Notice the qualitative difference in the shape of $C_\mathrm{dr}$
in the two subregions. In the upper subregion the portion $C_\mathrm{dr}^+$ of $C_\mathrm{dr}$ in the
positive quadrant has two points of inflection (it is concave to the left of the first
and to the right of the second and convex in between); passing into the lower subregion, one finds that the
two points of inflection merge and disappear, so that $C_\mathrm{dr}^+$ becomes concave.
The points $(\pm \ell_1^{(1)},0)$ approach the origin as one passes through $D_2$ from below to above, as does the left point of
inflection on $C_\mathrm{dr}^+$. In the region above $D_2$ the branches of $C_\mathrm{dr}$ intersect the $\ell_1$
axis vertically at the points $(\pm \ell_1^{(2)},0)$, while
in a neighbourhood of $(0,0)$ they have the limiting behaviour
$$
\ell_2^2\sim \left(\gamma^2\beta-1\right)\ell_1^2
$$
as $\ell_1 \to 0$ and therefore make angles
$\pm \arctan\sqrt{\gamma^2\beta-1}$ with the $\ell_1$ axis at the origin. The subcurve $C_\mathrm{dr}^+$ has a single
point of inflection, to the left and right of which it is respectively convex and concave. 

For given $\nu$, $\theta_1$ and $\theta_2$ the lines
$S_k$ in the $(\ell_1,\ell_2)$-plane are parallel,
equidistant and form an angle $\theta_2$ with
the positive $\ell_2$-axis. They intersect the line
\[T=\{(\ell_1,\ell_2)\in {\mathbb R}^2\;:\;
\ell_1=\sin\theta_1\,a,\
\ell_2=-\cos\theta_1\,a,\ a\in\R\}\]
(which passes through the origin and makes an angle $\theta_1$ with the positive
$\ell$-axis) at the points $P_k=(\sin\theta_1\,k\nu_0,
-\cos\theta_1\,k\nu_0)$, $k\in\Z$  (see Figure~\ref{Intersections}(b)).
The number of points in the set $S_0\cap C_\mathrm{dr}$ depends only upon
$(\beta,\gamma)$, which determines the shape of $C_\mathrm{dr}$, and $\theta_2$, which determines
the slope of each line $S_k$. Furthermore, for fixed $\beta$, $\gamma$ and $\theta_2$
the number of points in the sets $S_k\cap C_\mathrm{dr}$, $k=\pm1,\pm2,\dots$
depends only upon $\nu_0$, which determines the distance between the lines $S_k$.
At each fixed point of the $(\beta,\gamma)$-parameter plane the number of purely
imaginary eigenvalues of the linear operator $L$ therefore depends
upon the two parameters $\theta_2$ and $\nu_0$; the third parameter
$\theta_1$, which specifies the slope of the line $T$, influences only the
values of these eigenvalues and their relative positions
on the imaginary axis: the imaginary part of a purely imaginary eigenvalue corresponding
to an intersection of $S_k$ and $C_\mathrm{dr}$ is the value of $S_0$ in the $(S_0,T)$-coordinate
system at the intersection (the signed distance between the intersection and the point $P_n$).
The geometric multiplicity of the eigenvalue \(\ii s\) is given by the number of distinct lines in the family \(S_k\) that intersect 
\(C_{\text{dr}}\) at this parameter value, and a tangent intersection between \(S_k\) and \(C_{\text{dr}}\) indicates that each eigenvector in mode \(k\) has an associated Jordan chain of length at least \(2\).
Finally, notice that the sets
$S_k \cap C_\mathrm{dr}$ and $S_{-k} \cap C_\mathrm{dr}$ have the same cardinality:
the purely imaginary number $\ii s$ is a mode $k$ eigenvalue
if and only if the purely imaginary number $-\ii s$ is a mode $-k$ eigenvalue.

\begin{figure}[h!]\centering
\includegraphics{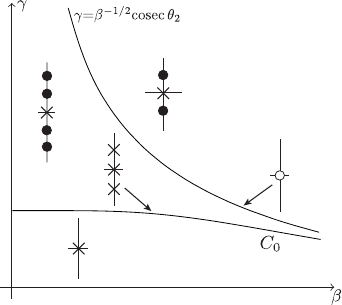}
\caption{Mode $0$ eigenvalues in the $(\beta,\gamma)$-parameter plane. Solid dots, crosses and hollow dots represent
eigenvalues with Jordan chains of length $1$, $2$ and $4$; the zero eigenvalue becomes essential spectrum at $\beta=0$.}
\label{bifdiagram}
\end{figure}

Figure \ref{bifdiagram} illustrates how the purely imaginary mode $0$ eigenvalues depend upon $(\beta,\gamma)$,
nonzero pairs $\pm \ii s$ of which satisfy
$$
(1+s^4)|s|-\frac{\gamma^2\sin^2 \theta_2 s^2}{\tanh(\beta^{-1}|s|)}=0.
$$
The delimiting curves are $C_0$ and $\{(\beta,\beta^{-1/2} \cosec \theta_2): \beta>0\}$,
points of which are associated with respectively non-zero eigenvalues $\pm \ii a \beta_0(a)$ with a Jordan chain of length $2$
and a zero eigenvalue with a Jordan chain of length $4$.

\subsection{Parameter selection}

We now choose $\beta$, $\gamma$ and $\theta_2$ such that $S_0$ does not intersect $C_\mathrm{dr}$ (so that $(\beta,\gamma)$ lies in the region
above the curve $C_0$ in Figure \ref{bifdiagram}), and $\nu_0$ and $\theta_1$ such that $S_1$ and $S_{-1}$ each intersect $C_\mathrm{dr}$
in points with coordinates $(\pm s,\nu_0)$ and $(\pm s,-\nu_0)$ in the $(S_0,T)$-coordinate system, while $S_k$ does not intersect 
$C_\mathrm{dr}$ for $k = \pm 2$, $\pm 3$, \ldots (see Figure \ref{scenario}). In this configuration
$L$ has two mode $1$ eigenvalues $\pm \ii s$,
two mode $-1$ eigenvalues $\pm \ii s$, so that $\pm \ii s$ are geometrically and algebraically double eigenvalues of $L$ with eigenvectors
$\hat{v}_{1,s}\ee^{\ii z}$, $\hat{v}_{-1,s}\ee^{-\ii z}$ and $\hat{v}_{1,-s}\ee^{\ii z}$, $\hat{v}_{-1,-s}\ee^{-\ii z}$, and, for $\beta>0$, a mode $0$ zero eigenvalue with a Jordan chain
$\hat{f}_1$, $\hat{f}_2$ of length $2$. Note that
$L^{\mu_1}$ has two mode $1$ eigenvalues $\ii s_1^{\mu_1}$, $-\ii s_{-1}^{\mu_1}$ and
two mode $-1$ eigenvalues $\ii s_{-1}^{\mu_1}$, $-\ii s_1^{\mu_1}$ which satisfy
\begin{equation}
g(s_{\pm 1}^{\mu_1},\nu_0+\mu_1,\pm 1)=0,
\label{ievals with mu}
\end{equation}
where
\begin{align*}
g(t,\nu,\pm 1)&=(1+(\sigma_{\pm 1})^4)\sigma_{\pm 1} - \frac{\gamma^2(\pm\nu^2 \sin \theta_1+t \sin \theta_2)^2}{\tanh(\beta^{-1}\sigma_{\pm 1})}, \\
(\sigma_{\pm 1})^2 &= t^2 \pm 2\nu t \cos(\theta_1-\theta_2)+\nu^2
\end{align*}
and $s_{\pm 1}^0=s$ (the corresponding eigenvectors are given by $\mathrm{d}G[0](v_{1,\pm s_{\pm}^{\mu_1}})$ and
$\mathrm{d}G[0](v_{-1,\pm s_{\pm}^{\mu_1}})$, where $v_{1,\pm s_{\pm}^{\mu_1}}$, $v_{-1,\pm s_{\pm}^{\mu_1}}$ are defined by the right-hand side of equation \eqref{evec} with $\nu_0$ replaced by $\nu_0+\mu_1$). Differentiating \eqref{ievals with mu} with respect to $\mu_1$, we obtain the formulae
\begin{equation*}
\frac{\mathrm{d}}{\mathrm{d}\mu_1}s_{\pm 1}^{\mu_1}=-\frac{g_{\nu}(s_{\pm1}^\mu,\nu_0+\mu,\pm1)}{g_t(s_{\pm 1}^\mu,\nu_0+\mu,\pm 1)}
\end{equation*}
which imply that
\begin{align*}
\frac{\mathrm{d}}{\mathrm{d}\mu}(s_{1}^{\mu_1}-s_{-1}^{\mu_1})\vert_{\mu_1=0}&=-\frac{g_{\nu}(s,\nu_0,1)}{g_t(s,\nu_0,1)}+\frac{g_{\nu}(s,\nu_0,-1)}{g_t(s,\nu_0,-1)}\\
&=\frac{1}{g_t(s,\nu_0,1)g_t(s,\nu_0,-1)}\left\vert\begin{array}{cc}
g_t(s,\nu_0,1) &g_t(s,\nu_0,-1)\\
g_{\nu}(s,\nu_0,1) & g_{\nu}(s,\nu_0,-1)
\end{array}\right\vert.
\end{align*}
It follows that $\frac{\mathrm{d}}{\mathrm{d}\mu}(s_{1}^\mu-s_{-1}^\mu)\vert_{\mu=0}= 0$ if and only if $\nabla g(s,\nu_0,1),\ \nabla g(s,\nu_0,-1)$ are parallel, in other words if and only if the solution curves of 
\begin{align*}
g(t,\nu_,1)&=0,\\
g(t,\nu,-1)&=0,
\end{align*}
intersect tangentially at the point $(s,\nu_0)$ in the $(t,\nu)$-plane. This observation indicates that generically $\frac{\mathrm{d}}{\mathrm{d}\mu}(s_{1}^\mu-s_{-1}^\mu)\vert_{\mu=0}\neq 0$.

Noting that
$$\hat{v}_{1,-s}\ee^{\ii z}=\overline{\hat{v}_{-1,s}\ee^{-\ii z}}=R(\hat{v}_{-1,s}\ee^{-\ii z}), \qquad
\hat{v}_{-1,-s}\ee^{-\ii z}=\overline{\hat{v}_{1,s}\ee^{\ii z}}=R(\hat{v}_{1,s}\ee^{\ii z})$$
and $R\hat{f}_1=-\hat{f}_1$, $R\hat{f}_2=\hat{f}_2$, we find that there exists parameters such that hypotheses (H3)--(H6) are satisfied (with (H4)(ii) for $\beta=0$
and (H4)(iii) for $\beta>0$). Hypothesis (H7) follows from the observations that $(G^{\mu_1})^{-1}(\hat{f}_1) = \hat{f}_1$ and
that $\hat{J}^{\mu_1}(\hat{v})$, $\hat{H}^{\mu_1}(\hat{v})$ depend upon $\Gamma$ and $\zeta$ only through $\Gamma_y$
and $\zeta-\gamma\sin\theta_1\Gamma|_{y=0}$, while hypothesis (H8) is verified in Section \ref{Iooss condition} below.

Applying Theorem \ref{LCT} thus yields a family of doubly periodic waves whose periodic cells are small perturbations of the basic periodic cell defined by the
periods $2\pi/\nu_0$, $2\pi/s$ and angle $\theta_2-\theta_1$\linebreak between the periodic directions. The following lemma shows that it is in fact
possible to choose the basic periodic cell (and value of $\beta$)
arbitrarily and adjust the value of $\gamma$ and the angle $\theta_1$ to ensure that Theorem \ref{LCT} applies in this configuration (under the additional hypothesis that $S_k$ does not intersect  $C_\mathrm{dr}$ for $k \neq \pm 1$) .

\begin{lemma}
Choose $\beta$, $s$, $\nu_0$ and $\theta_2-\theta_1$. There exist $\theta_1$ and $\gamma$ such that $S_1$ and $S_{-1}$ each intersect $C_\mathrm{dr}$
in points with coordinates $(\pm s,\nu_0)$ and $(\pm s,-\nu_0)$ in the $(S_0,T)$-coordinate system.
\end{lemma}

\begin{proof}
The lines $S_1$ and $S_{-1}$ intersect $C_\mathrm{dr}$ at respectively $(s, \nu_0)$ and $(s,-\nu_0)$ if and only if 
\begin{align*}
\gamma^2 b_1^2 &= \tanh(\beta^{-1} \sigma_1)(1+\sigma_1^4)\sigma_1, \\
\gamma^2 b_{-1}^2 &= \tanh(\beta^{-1} \sigma_{-1})(1+\sigma_{-1}^4)\sigma_{-1},
\end{align*}
and in this case they also intersect $C_\mathrm{dr}$ at respectively $(-s,-\nu_0)$ and $(-s,\nu_0)$ (see the remarks at the end of Section \ref{pis}).
Let $e_x$, $e_z$ and $i$ be unit vectors in the $x$, $z$ and $x_1$ directions (see Figure \ref{intro - periodic domain}) and set
\begin{equation}
\ell_1= s e_z+\nu_0e_x \qquad \ell_2 =  - s e_z+\nu_0e_x, \label{Definition of lines}
\end{equation}
so that $\sigma_1^2 = |\ell_1|^2$, $\sigma_{-1} = |\ell_2|^2$, $i \cdot \ell_1 = b_1$, $i \cdot \ell_2 = b_{-1}$; in this notation our task is to find a solution of the equations
\begin{align}
(v \cdot \ell_1)^2 &= \tanh(\beta^{-1} |\ell_1|)(1+ |\ell_1|^4) |\ell_1|, \label{Lines 1}\\
(v \cdot \ell_2)^2 &= \tanh(\beta^{-1} |\ell_2|)(1+ |\ell_2|^4) |\ell_2| \label{Lines 2},
\end{align}
of the form $v = \gamma i$. Let $\psi$ denote the angle between $\ell_1$ and $\ell_2$; equation \eqref{Definition of lines} shows that rotating $e_x$ and $e_z$ through the same angle $\theta$,
that is changing $\theta_1$ and $\theta_2$ by $\theta$, causes $\ell_1$ and $\ell_2$ to rotate through $\theta$ and thus does not change $\psi$.

Observe that the equations
\begin{align*}
v \cdot n_1 = \pm \alpha_1, \\
v \cdot n_2 = \pm \alpha_2,
\end{align*}
where $n_1$, $n_2 \in {\mathbb R}^2$ are linearly independent and $\alpha_1$, $\alpha_2 > 0$, represent two pairs of parallel lines which intersect in four points. Each of these points $v$ satisfies
$$|v| = \frac{|\alpha_2 n_1 \pm \alpha_1 n_2|}{|n_1||n_2|},$$
so that $|v|$ depends only upon $\alpha_1$, $\alpha_2$, $|n_1|$, $|n_2|$ and the angle between $n_1$ and $n_2$. Using this result we find that the solution set to equations
\eqref{Lines 1}, \eqref{Lines 2} consists of four points. Let $v$ be one of these points, and note that
$|v|$ depends only upon $|\ell_1|$, $|\ell_2|$ and $\psi$
(since the right-hand sides of \eqref{Lines 1}, \eqref{Lines 2} depend only upon $|\ell_1|$ and $|\ell_2|$.) By rotating $\ell_1$ and $\ell_2$ through a suitably chosen angle $\theta$, that is changing $\theta_1$ and $\theta_2$ by $\theta$, we can arrange that $v=|v|i$ and hence $v=\gamma i$ by setting $\gamma=|v|$.

\end{proof}

\subsection{Verification of hypothesis (H8)} \label{Iooss condition}
It remains to verify that hypothesis (H8) is satisfied when $\beta=0$. The condition is that
the equation
\begin{equation}\label{solvability_nonlinear_1}
Lv^\dag=\hat{J}^0(0)^{-1}[\hat{J}^{\mu_1}(\hat{v})\big((\kappa_0+\mu_2)\hat{v}_\tau-\hat{v}_{\mathrm{H}}^{\mu_1}(\hat{v})\big)
+\hat{J}^0(0)L\hat{v}]_0,
\end{equation}
has a unique solution $v^\dag \in X$ which depends smoothly upon $(\hat{v},\mu_1,\mu_2) \in {\mathcal U} \times \Lambda_1 \times
\Lambda_2$.
Since
\begin{equation*}
\hat{J}^0(0)^{-1}[\hat{J}^0(0)L\hat{v}]_{0}=L[\hat{v}]_0
\end{equation*}
one can rewrite equation \eqref{solvability_nonlinear_1} as
\begin{equation}\label{solvability_nonlinear_2}
L\big(v^\dag-[\hat{v}]_0\big)=\hat{J}^0(0)^{-1}[\hat{J}^{\mu_1}(\hat{v})\big((\kappa_0+\mu_2)\hat{v}_\tau-\hat{v}_{\mathrm{H}}^{\mu_1}(\hat{v})\big)]_0.
\end{equation}
In view of equation \eqref{solvability_nonlinear_2} and Lemma \ref{lemma_spectrum}(iii) our task is to show that
\begin{equation}
(\hat{v},\mu_1,\mu_2) \mapsto \int_{-\infty}^y \bl[\hat{w}_\Gamma]_0\br_0, \qquad
(\hat{v},\mu_1,\mu_2) \mapsto \int_{-\infty}^y \int_{-\infty}^t \bl[\hat{w}_\Gamma]_0\br_0 \ds\dt,
\label{Integral mappings}
\end{equation}
where
$$\hat{w}=\hat{J}^{\mu_1}(\hat{v})\big((\kappa_0+\mu_2)\hat{v}_\tau-\hat{v}_{\mathrm{H}}^{\mu_1}(\hat{v})\big),$$
are smooth mappings ${\mathcal U} \times \Lambda_1 \times \Lambda_2 \rightarrow L^2(-\infty,0)$ and that
\begin{equation}
\int_{-\infty}^0\bl[\hat{w}_\Gamma]_0\br_0\ \mathrm{d}y=-\bl[\hat{w}_\zeta]_0\br_0\gamma\sin\theta_2\label{comp_hydroelastic}
\end{equation}
for each $(\hat{v},\mu_1,\mu_2) \in {\mathcal U} \times \Lambda_1 \times \Lambda_2$. Here we use the symbols
$[\cdot]_0$ and $\bl \cdot\br_0$ to denote the projections onto the $0$th Fourier modes in the $\tau$ and $z$ variables respectively.
This task is accomplished in Lemma \ref{solvability_hat_w} below with the help of the following auxiliary result.

\begin{proposition}\label{solvability_v_prop}
The formulae
$$
(v,\mu_1) \mapsto \int_{-\infty}^y \bl[(v_{\mathrm{H}}^{\mu_1}(v))_\Psi]_0\br_0 \dt, \qquad (v,\mu_1) \mapsto \int_{-\infty}^y \int_{-\infty}^t \bl[(v_{\mathrm{H}}^{\mu_1}(v))_\Psi]_0\br_0 \ds\dt
$$
define analytic mapppings ${\mathcal U} \times \Lambda_1 \rightarrow L^2(-\infty,0)$ and
$$\int_{-\infty}^0\bl[(v_{\mathrm{H}}^{\mu_1}(v))_\Psi]_0\br_0 \dy=\bl[(v_{\mathrm{H}}^{\mu_1}(v))_\eta]_0\br_0\gamma\sin \theta_2$$
for each $(v,\mu_1) \in {\mathcal U} \times \Lambda_1$.
\end{proposition}
\begin{proof}
It follows from \eqref{Ham eq 6} that
$$\bl[(v_{\mathrm{H}}^{\mu_1}(v))_\Psi]_0\br_0=\bl[(g_0^{\mu_1}(v))_y]_0\br_0,$$
where\pagebreak
\begin{align*}
g_0^{\mu_1}(v)&=\Phi_y-K_1(\eta)\eta\Phi_y-K_1(\eta)\Psi\rho
-\frac{K_1(\eta)}{K_2(\eta)}(\nu_0+\mu_1)^2\sin^2(\theta_1-\theta_2)(\Phi_z-K_1(\eta)\eta_z\Phi_y)\eta_z \\
& \qquad\mbox{}-(\nu_0+\mu_1)\cos(\theta_1-\theta_2)K_1(\eta)\eta_z\Psi \\
&=\Phi_y-\Big(K_2(\eta)\eta\Phi_y-K_2(\eta)\Psi\rho
-(\nu_0+\mu_1)^2\sin^2(\theta_1-\theta_2)(\Phi_z-K_1(\eta)\eta_z\Phi_y)\eta_z \\
& \qquad\qquad\quad\mbox{}-(\nu_0+\mu_1)\cos(\theta_1-\theta_2)K_2(\eta)\eta_z\Psi\Big)\ee^y
\end{align*}
and $v=(\eta,\rho,\Phi,\zeta,\xi,\Psi)$.
Observing that $(v,\mu_1) \mapsto g_0^{\mu_1}(v)$ maps $U \times \Lambda_1$ analytically into $L^2(-\infty,0)$ (see Proposition \ref{Analytic stuff}) and the same is true
of $(v,\mu_1) \mapsto \int_{-\infty}^y g_0^{\mu_1}(v)\dt$ because\linebreak
$u \mapsto \int_{-\infty}^y u(t) \ee^t \dt$ belongs to ${\mathcal L}(L^2(-\infty,0))$, we conclude that
$(v,\mu_1) \mapsto \int_{-\infty}^y \bl[(v_{\mathrm{H}}^{\mu_1}(v))_\Psi]_0\br_0 \dt$ and $(v,\mu_1) \mapsto \int_{-\infty}^y \int_{-\infty}^t \bl[(v_{\mathrm{H}}^{\mu_1}(v))_\Psi]_0\br_0 \ds\dt$
map ${\mathcal U} \times \Lambda_1$ analytically into $L^2(-\infty,0)$.
Finally
$$\int_{-\infty}^0\bl[(v_{\mathrm{H}}^{\mu_1}(v))_\Psi]_0\br_0 \dy=\bl[\rho]_0\br_0\gamma\sin \theta_2=\bl[(v_{\mathrm{H}}^{\mu_1}(v))_\eta]_0\br_0\gamma\sin \theta_2$$
because of \eqref{BC 1}, \eqref{BC 2} and \eqref{Ham eq 1}.
\end{proof}

\begin{lemma} \label{solvability_hat_w}
The formulae \eqref{Integral mappings} define analytic mappings ${\mathcal U} \times \Lambda_1 \times \Lambda_2 \rightarrow L^2(-\infty,0)$
and the formula \eqref{comp_hydroelastic} is satisfied for each $(\hat{v},\mu_1,\mu_2) \in {\mathcal U} \times \Lambda_1 \times \Lambda_2$.
\end{lemma}
\begin{proof}
We first note that
\begin{equation}\label{Gamma_inverse1}
\widetilde{\mathrm{d}G}\vphantom{}^{\mu_1}[v]^\ast(\hat{w})=J\big((\kappa_0+\mu_2)v_\tau-v_{\mathrm{H}}^{\mu_1}(v)\big),
\end{equation}
where $v=K^{\mu_1}(\hat{v})$, and using the explicit formulae for $\widetilde{\mathrm{d}G}\vphantom{}^{\mu_1}[v]^\ast$ appearing in the proof of Proposition \ref{adjoint_proposition}, we find that the $\zeta$- and
$\Gamma$-components of equation \eqref{Gamma_inverse1} are
\begin{align}
\hat{w}_\zeta &= (\kappa_0+\mu_2)\frac{\mathrm{d}}{\mathrm{d}\tau}v_\eta-\left(v_{\mathrm{H}}^{\mu_1}(w)\right)_\eta, \label{zeta_component} \\
\hat{w}_\Gamma-(g_1^{\mu_1}(v,\hat{w}))_y-(g_2(v,\hat{w}))_z&=-(\kappa_0+\mu_2)\frac{\mathrm{d}}{\mathrm{d}\tau}v_\Psi+\left(v_{\mathrm{H}}^{\mu_1}(v)\right)_\Psi, \label{Gamma_inverse2}
\end{align}
where 
\begin{align*}
g_1^{\mu_1}(v,\hat{w})&=\bigg(K_2(v_\eta)K_3v_\eta-\frac{K_1^2(v_\eta)}{K_2(v_\eta)}(\nu_0+\mu_1)^2\sin^2(\theta_1-\theta_2)(v_\eta)_z^2\bigg)\Delta^{-1}((\hat{w}_\Gamma)_y) \\
&=\bigg(K_2(v_\eta)v_\eta-K_1(v_\eta)(\nu_0+\mu_1)^2\sin^2(\theta_1-\theta_2)(v_\eta)_z^2\bigg)\ee^y\Delta^{-1}((\hat{w}_\Gamma)_y), \\
g_2(v,\hat{w})&=\bigg(\frac{K_1(v_\eta)}{K_2(v_\eta)}(\nu_0+\mu_1)^2\sin^2(\theta_1-\theta_2)(v_\eta)_z\bigg)\Delta^{-1}((\hat{w}_\Gamma)_y);
\end{align*}
note in particular that $(\hat{v},\mu_1) \mapsto g_1^{\mu_1}(v,\hat{w})$ maps $\hat{U} \times \Lambda_1$ analytically into $L^2(-\infty,0)$ (see Proposition \ref{Analytic stuff}) and the same is true of
$(\hat{v},\mu_1) \mapsto \int_{-\infty}^y g_1^{\mu_1}(v,\hat{w})\dt$ because
$u \mapsto \int_{-\infty}^y u(t) \ee^t \dt$ belongs to ${\mathcal L}(L^2(-\infty,0))$.

Equation \eqref{Gamma_inverse2} implies that 
\begin{equation*}
\bl[\hat{w}_\Gamma]_0\br_0=\bl[ (g_1^{\mu_1}(v,\hat{w}))_y]_0\br_0+\bl[\left(v_{\mathrm{H}}^{\mu_1}(v)\right)_\Psi]_0\br_0,
\end{equation*}
and it follows from this identity and Proposition \ref{solvability_v_prop} that $(\hat{v},\mu_1,\mu_2) \mapsto \int_{-\infty}^y \bl[\hat{w}_\Gamma]_0\br_0\dt$ and\linebreak
$(\hat{v},\mu_1,\mu_2) \mapsto \int_{-\infty}^y \int_{-\infty}^t \bl[\hat{w}_\Gamma]_0\br_0 \ds\dt$ map ${\mathcal U} \times \Lambda_1 \times \Lambda_2$ analytically into $L^2(-\infty,0)$.
Furthermore, the calculation
\begin{align*}
\int_{-\infty}^0\bl[\hat{w}_\Gamma]_0\br_0\dy&=\underbrace{\big[\bl[ g_1^{\mu_1}(w,\hat{w})]_0\br_0\big]_{-\infty}^0}_{\displaystyle =0}+\int_{-\infty}^0\bl[\left(v_{\mathrm{H}}^{\mu_1}(w)\right)_\Psi]_0\br_0\dy\\
&=\bl\left(v_{\mathrm{H}}^{\mu_1}(w)\right)_\eta]_0\br_0\gamma\sin\theta_2\\
&=-\bl[\hat{w}_\zeta]_0\br_0\gamma\sin\theta_2,
\end{align*}
shows that \eqref{comp_hydroelastic} is also satisfied; here we have used Proposition \ref{solvability_v_prop} and the facts that
$$\bl[\hat{w}_\zeta]_0\br_0=-\bl[\left(v_{\mathrm{H}}^{\mu_1}(w)\right)_\eta]_0\br_0$$
(see equation \eqref{zeta_component}) and
$\Delta^{-1}((\hat{w}_\Gamma)_y)\big|_{y=0}=0$.
\end{proof}
\section{Declarations}
\noindent \textbf{Funding}\\
Part of this work was carried out while D.N was supported by a grant from the Knut and Alice Wallenberg foundation.\newline

\noindent \textbf{Competing interests}\\
The authors have no competing interests to declare that are relevant to the content of this article.
\appendix

\bibliographystyle{standard}
\bibliography{mdg}

\end{document}